\documentclass[a4paper,12pt]{article}

\usepackage[active]{srcltx}
\usepackage{hyperref}

\usepackage{amsmath,amsthm,amssymb,mathrsfs,latexsym,amsfonts}
\usepackage{graphicx,psfrag,epsfig}
\usepackage{bbm}
\usepackage[english]{babel}
\usepackage[latin1]{inputenc}
\usepackage{a4wide}

\newtheorem{theorem}{Theorem}
\newtheorem{lemma}[theorem]{Lemma}
\newtheorem{proposition}[theorem]{Proposition}

\newtheorem{corollary}[theorem]{Corollary}

\theoremstyle{definition}

\newtheorem{remark}{\it Remark}
\newtheorem{example}[theorem]{Example}

\newcounter{paraga}[section]

\newtheorem{Main}{Theorem}

\def\MP{\,{<\hspace{-.5em}\cdot}\,}

\def\PM{\,{\cdot\hspace{-.3em}<}\,}

\def\EP{\,{=\hspace{-.2em}\cdot}\,}

\def\N{\mathbb N}
\def\C{\mathbb C}
\def\Q{\mathbb Q}
\def\R{\mathbb R}
\def\T{\mathbb T}

\def\Z{\mathbb Z}

\begin{document}

\begin{titlepage}
  \title{\LARGE{\textbf{KAM, $\alpha$-Gevrey regularity\\
        and the $\alpha$-Bruno-R{\"u}ssmann condition}}}%
  \author{Abed Bounemoura and Jacques F{\'e}joz}
\end{titlepage}

\maketitle

\begin{abstract}
  We prove a new invariant torus theorem, for $\alpha$-Gevrey smooth
  Hamiltonian systems, under an arithmetic assumption which we call
  the $\alpha$-Bruno-R{\"u}ssmann condition, and which reduces to the
  classical Bruno-R{\"u}ssmann condition in the analytic category. Our
  proof is direct in the sense that, for analytic Hamiltonians, we
  avoid the use of complex extensions and, for non-analytic
  Hamiltonians, we do not use analytic approximation nor smoothing
  operators. Following Bessi, we also show that if a slightly weaker
  arithmetic condition is not satisfied, the invariant torus may be
  destroyed. Crucial to this work are new functional estimates in the
  Gevrey class.
\end{abstract}


\section{Introduction}\label{s1}

\subsection{The general question}

We consider small perturbations of an integrable Hamiltonian system,
defined by
\begin{equation*}
\dot{q}=\nabla_p H(q,p), \quad \dot{p}=- \nabla_q H(q,p)
\end{equation*}
where $H$ is a Hamiltonian of the form
\[ H(q,p)=h(p) + \epsilon f(q,p), \quad (q,p) \in \T^n \times \R^n,
\quad 0 \leq \epsilon <1\]
where $n\geq 2$, $\T^n=\R^n / \Z^n$, $\omega_0 = \nabla h(0) \in\R^n$,
and $\nabla^2 h(0) \in M_n(\R)$ is non-degenerate. When $\epsilon=0$,
the torus $\mathcal{T}_0$ of equation $p=0$ is invariant and
quasi-periodic of frequency $\omega_0$. The general question we are
interested in is the persistence of this torus for $\epsilon \neq 0$
sufficiently small : does there exist a torus $\mathcal{T}_\epsilon$
which is invariant and quasi-periodic of frequency $\omega_0$ and
which converges (in a suitable sense) to $\mathcal{T}_0$ as $\epsilon$
goes to zero?

This question was answered positively by Kolmogorov in his
foundational paper~\cite{Kol54} under the assumption that $H$ is
real-analytic and $\omega_0$ is a $\tau$-Diophantine vector
($\tau \geq n-1$): there exists $\gamma>0$ such that for all
$k \in \Z^n\setminus \{0\}$, $|k\cdot\omega_0|\geq \gamma|k|^{-\tau}$.
As a conclusion, the perturbed torus is real-analytic. It became clear
that a regularity assumption on $H$ and an arithmetic condition on
$\omega_0$ were necessary, and then further works investigate the
interplay between the analysis and the arithmetic.

It was certainly a remarkable contribution of Moser (see \cite{Mos62})
to realize that the question can also be answered for Hamiltonians
which are only finitely differentiable. More precisely (see
\cite{Sal04}), if $\omega_0$ is $\tau$-Diophantine and if $H$ is of
class $C^r$, with $r>2(\tau+1)$, then the torus persists and it is of
class $C^{r'+\tau+1}$ for any $r'<r-2(\tau+1)$. If $H$ is smooth, that
is $C^{\infty}$, there is no restriction on $\tau$ and the perturbed
torus is smooth. It follows from a recent result of
Cheng-Wang~\cite{CW13} (which uses an idea of Bessi~\cite{Bes00}) that
the result is false if $H$ is of class $C^r$, with $r<2(\tau+1)$. Thus
in the finitely differentiable or smooth case, one may consider this
Diophantine condition as essentially optimal.

In the real-analytic setting, the Diophantine condition is not
necessary. Indeed, it is sufficient to assume that $\omega_0$
satisfies the weaker Bruno-R{\"u}ssmann condition (see \S\ref{sec:aBR} for
a definition), as was first proved by R{\"u}ssmann in~\cite{Rus01}; an
equivalent condition was actually introduced earlier by
Bruno~\cite{Bru71},~\cite{Bru72} in a different but related small
divisors problem, the Siegel linearization problem. The necessity of
this condition turns out to be a more subtle problem. In the Siegel
problem, it is optimal in dimension one (this is a celebrated result
of Yoccoz~\cite{Yoc88},~\cite{Yoc95}) but in higher dimension it is
unknown. In the Hamiltonian problem we are considering here, the only
general result we are aware of is due to Bessi~\cite{Bes00} (extending
an earlier result of Forni~\cite{For94} for twist maps of the annulus)
in which a torus with a frequency not satisfying a slightly weaker
condition can be destroyed by an arbitrary small analytic
perturbation. This leaves open the possibility of slightly improving
the Bruno-R{\"u}ssmann condition.

\subsection{Main results of the paper} 

Real-analytic functions are characterized by a growth of their derivatives of order $s^{-|k|}k!$ for some analyticity width $s>0$; in the periodic case, this is equivalent to a decay of Fourier coefficients of order $e^{-s|k|}$. Given a real parameter $\alpha \geq 1$, allowing a growth of the derivatives of order $s^{-|k|}k!^\alpha$ or, equivalently, a decay of Fourier coefficients of order $e^{-\alpha s|k|^{1/\alpha}}$, one is lead to consider $\alpha$-Gevrey functions, which thus corresponds to real-analytic functions when $\alpha=1$. Since the introduction by Gevrey of the class of functions now baring his name (\cite{Gev18}), there has been a huge amount of works on Gevrey functions, mainly for PDEs, but also more recently in other fields, including dynamical systems (see \S\ref{sec:related} for some related works in dynamical systems dealing with Gevrey regularity).     
 
In this paper, we study the persistence when $\epsilon$ is small of the torus $\mathcal{T}_0$, as a Gevrey
quasiperiodic invariant torus $\mathcal{T}_\varepsilon$, under the assumptions that $H$ itself has Gevrey regularity. The only general result so far is due to Popov~\cite{Pop04} who proved that the latter holds true if $\omega$ satisfies a Diophantine condition. This result, the proof of which uses analytic approximation, extends the result of Kolmogorov when $\alpha=1$ but not the one of R{\"u}ssmann: clearly one would expect an arithmetic condition which does depend on $\alpha$ and that reduces to the Bruno-R{\"u}ssmann condition when $\alpha=1$.  

The main result of the paper is to solve this persistency problem, assuming that the frequency $\omega_0$
satisfies some arithmetic condition which we call the
$\alpha$-Bruno-R{\"u}ssmann condition, which is weaker than the Diophantine condition and agrees with the Bruno-R{\"u}ssmann condition when $\alpha=1$. This is the content of
Theorem~\ref{classicalKAM}; Theorem~\ref{isoKAM} and Theorem~\ref{timeKAM} deal respectively with the iso-energetic and time-periodic versions. We will also state and prove a Gevrey analogue of Arnold's normal form theorem for vector fields on the torus (Theorem~\ref{KAMvector}). Theorem~\ref{KAMparameter} is a more
precise, quantitative statement, with parameters, which does not
require non-degeneracy, and from which Theorem~\ref{classicalKAM} and Theorem~\ref{KAMvector} follow. We also notice that Bessi's ideas~\cite{Bes00} may be adapted to the Gevrey setting, to provide a necessary arithmetic condition for the invariant torus to persist (Theorem~\ref{destruction}). The so-obtained condition fails to agree with the sufficient condition of Theorem~\ref{classicalKAM} and, as in the analytic case, it remains open to determine the optimal condition. Finally, we will also give discrete versions of Theorem~\ref{classicalKAM} and Theorem~\ref{KAMvector}, which are, respectively, Theorem~\ref{KAMmap1} and Theorem~\ref{KAMmap2}.  

When a Hamiltonian is not real analytic, it is often the case that there is still some control on its derivatives and that it has Gevrey regularity. This may happen for example for the restriction of an analytic Hamiltonian restricted to a Gevrey, symplectic, central manifold. Technically, Gevrey regularity luckily extends the well-behaved analytic regularity in KAM theory: the effect of small denominators in Fourier series reduces to decreasing the ``Gevrey width'' $s$, the analogue of the analyticity width. This makes it possible to adapt Kolmogorov's proof of his invariant torus theorem without using analytic approximations or smoothing operators as in the smooth setting. Yet there are two issues one needs to solve.

The first and main issue is that the estimates needed in the general
problem of perturbation theory were missing. This is why we provide an
appendix with an adequate choice of norms and spaces, together with
the estimates needed in our proof. In particular,
Proposition~\ref{composition} provides a ``geometric'' estimate of the
composition of two Gevrey functions, in which the loss of Gevrey width
is arbitrarily small when composing a function to the right by a
diffeomorphism close to the identity, in continuity with the
real-analytic setting. Starting with the work of Gevrey
itself~\cite{Gev18}, there have been many results concerning the
composition of Gevrey functions (see, for instance,
Yamanaka~\cite{Yam89}, Marco-Sauzin~\cite{MS02},
Cadeddu-Gramchev~\cite{Cadeddu:2003}, Popov~\cite{Pop04}) but none of
them allowed an arbitrarily small loss of width except in some
particular cases (the one-dimensional case and the analytic case). To
our knowledge, our composition result is new and may be of independent
interest.

The second and minor issue is that to reach a weak arithmetic condition, it is usually better not to solve exactly the cohomological equation but an approximate version of it, and hence one cannot proceed as in Kolmogorov's proof. The strategy of R{\"u}ssmann, that we could have tried to pursue here, consists in solving this equation not for the original perturbation but for a polynomial approximation of it. We will rather adopt the strategy of~\cite{BF13},~\cite{BF14} in which periodic approximations of the frequency are used and only cohomological equations associated to periodic vectors need to be solved: estimates on the solution are straightforward in this case, unlike the cohomological equation associated to a non-resonant vector.    

As a further remark concerning the proof, in invariant tori problems derivatives in the angle and action directions do not play the same role: in the analytic case it is customary to introduce anisotropic norms. However, as Theorem~\ref{KAMparameter} and its proof show, we can still get good estimates if we keep track separately of the sizes of various
terms in the expansion of the Hamiltonian with respect to the
actions: this turns out simpler than using anisotropic Gevrey norms. Such a feature is not present in dealing with linearization problem such as in Theorem~\ref{KAMvector}; a direct proof of the latter result would have been much simpler.

\subsection{Related results}\label{sec:related}

Apart from the work of Popov that we have already mentioned, there have been several works dealing with Gevrey regularity in a related context.

The first setting is the so-called Siegel-Sternberg linearization problem. Under a non-resonance condition, a formal solution to the conjugacy problem always exists and Sternberg proved that the solution is in fact smooth. In the analytic case, under the Bruno-R{\"u}ssmann condition the conjugacy is analytic; this arithmetic condition is thus sufficient but also necessary in (complex) dimension one (a result of Yoccoz we already mentioned). In the Gevrey setting, still under the Bruno-R{\"u}ssmann condition, Carletti-Marmi~\cite{CM00} and Carletti~\cite{Car03} have shown that the formal solution still has Gevrey growth (with the same Gevrey exponent); an interesting feature of their result is that allowing a worse Gevrey exponent for the formal solution, one can relax accordingly the arithmetic condition. All these results are actually valid for a class of ultra-differentiable functions that includes analytic and Gevrey functions. It was then proved by Stolovitch~\cite{Sto13} that this formal Gevrey solution actually give rise to a Gevrey smooth solution, and recently, P{\"o}schel~\cite{Pos17} gave a very general version of the Siegel-Sternberg theorem for ultra-differentiable functions that contains all the previous results (the smooth, analytic, Gevrey and ultra-differentiable cases). Let us mention that all these results do use stability by composition, but a precise composition result is not needed as they do not require to keep track of the width.  

In the analytic setting, the Siegel problem and the problem of the linearization of circle diffeomorphisms are solved under the same arithmetic condition~\cite{PM97}. But this may well be incidental, and, to our knowledge, it may well not be true in the Gevrey setting. The only result concerning Gevrey circle diffeomorphism we are aware of is due to Gramchev-Yoshino~\cite{Gramchev:1999}: they proved the linearization theorem under a condition which is weaker than the Diophantine condition but stronger than the $\alpha$-Bruno-R{\"u}ssmann condition (they actually introduce a condition equivalent to our $\alpha$-Bruno-R{\"u}ssmann condition and conjecture that the result should hold under this condition). To prove such a result, they use a composition result but in one dimension only; in this special case, as we already pointed out above, good composition estimates are known (see, for instance,~\cite{MS02}). As a matter of fact, Theorem~\ref{KAMmap2} (the discrete version of Theorem~\ref{KAMvector}) gives linearization of Gevrey torus diffeomorphism close to a translation under the $\alpha$-Bruno-R{\"u}ssmann condition, extending the result in~\cite{Gramchev:1999} (and giving a positive answer to their conjecture).

\subsection{Further results}

Let us describe some further results that could be achieved using the techniques of this paper. The literature on KAM theory is enormous and so there are many potential applications; we will only describe here some of those that may have some interest.

First, and more importantly, the technical estimates we derive in Appendix~\ref{app2} for Gevrey functions actually hold true for a larger class of ultra-differentiable functions that includes Gevrey (and thus analytic) functions as a particular case. This not only leads to a further extension of the KAM theorems we state and prove here, but also allows us to generalize other perturbative results such as the Nekhoroshev theorem (extending the result of~\cite{MS02} in the convex case and~\cite{Bou11} in the steep case). To keep this paper to a reasonable length, all these results will be derived in a subsequent article~\cite{BFb17}.

Then, our main result Theorem~\ref{classicalKAM} deals with the persistence of Lagrangian tori; KAM theory also deals with lower-dimensional tori (see, for instance,~\cite{Rus01} for a comprehensive treatment in the analytic case), and one may expect that our result extend to such a setting.

Finally, one may consider the problem of reducibility of quasi-periodic cocycles close to constant. In the analytic case, the Bruno-R{\"u}ssmann condition is sufficient, as was shown in~\cite{CM12}; in the $\alpha$-Gevrey case, the $\alpha$-Bruno-R{\"u}ssmann condition is sufficient. In fact, this setting is simpler from a technical point of view and our Gevrey estimates are not necessary to obtain such a result; one simply needs to go through the proof of~\cite{CM12}. A possible explanation for this is that for quasi-periodic cocycles, composition occur in a linear Lie group, thus only estimates for linear composition (product of matrices) are necessary and so everything boils down to good estimates for the product of two functions.

\subsection{Plan of the paper}

The plan of the paper is as follows.

In Section~\ref{sec:setting}, we describe precisely the setting,
namely we properly define the Gevrey norms we will use and the
$\alpha$-Bruno-R{\"u}ssmann condition. In Section~\ref{sec:main-results}
we state our main results:
\begin{itemize}
\item Theorem~\ref{classicalKAM} about the persistence of a torus in a
  non-degenerate Hamiltonian system under the $\alpha$-Bruno-R{\"u}ssmann
  condition;
\item Theorem~\ref{isoKAM}, the iso-energetic version of
  Theorem~\ref{classicalKAM};
\item Theorem~\ref{timeKAM}, the non-autonomous time-periodic version;
\item Theorem~\ref{destruction} about the destruction of a torus in
  the same context not assuming a condition weaker that the
  $\alpha$-Bruno-R{\"u}ssmann condition;
\item Theorem~\ref{KAMvector} about linearization of vector fields on
  the torus close to constant (we will also discuss necessary
  arithmetic conditions here, albeit in a restricted context);
\item Theorem~\ref{KAMmap1}, the discrete version of
  Theorem~\ref{classicalKAM}, about the persistence of a torus in a
  non-degenerate exact-symplectic map;
\item Theorem~\ref{KAMmap2}, the discrete version of
  Theorem~\ref{KAMvector}, about the linearization of diffeomorphisms
  of the torus close to a translation.
\end{itemize}
In Section~\ref{sec:KAMparam} we state Theorem~\ref{KAMparameter}, the
main technical result of this paper, which is a KAM theorem which do
not require non-degeneracy but depends on parameters. In
Section~\ref{sec:proofs}, we give the proof of
Theorems~\ref{classicalKAM} and~\ref{KAMvector}, assuming
Theorem~\ref{KAMparameter}. Section~\ref{sec:proof} contains the proof
of Theorem~\ref{KAMparameter}. Section~\ref{sec:bessi} contains the
proof of Theorem~\ref{destruction}, a straightforward extension of the
work of Bessi~\cite{Bes00}.

Finally, two appendices contain technical results. Appendix~\ref{app1}
provides various characterizations of the $\alpha$-Bruno-R{\"u}ssmann
condition. Appendix~\ref{app2}, which is absolutely crucial in this
work, provides estimates on Gevrey functions (and in particular our
composition result Proposition~\ref{composition}) which are use
throughout the paper.

\section{Setting}\label{sec:setting}

\subsection{Gevrey Hamiltonians}\label{sec:Gevrey}

Recall that $n\geq 1$ is an integer, $\T^n=\R^n / \Z^n$ and let
$B \subseteq \R^n$ be a bounded open domain containing the origin. For
a small parameter $\epsilon\geq 0$, we consider a Hamiltonian function
$H : \T^n \times B \rightarrow \R$ of the form
\begin{equation}\label{Ham1}
\begin{cases}\tag{$*$}
  H(q,p)= h(p) +\epsilon f(q,p), \\
  \nabla h(0):=\omega_0 \in \R^n.
\end{cases}  
\end{equation}
The Hamiltonian $h$ is \textit{non-degenerate} at the origin if the
matrix $\nabla^2 h(0)$ itself is non-degenerate. We shall assume that
the Hamiltonian $H$ is \textit{$\alpha$-Gevrey} on
$\T^n \times \bar{B}$, with $\alpha \geq 1$ and where $\bar{B}$
denotes the closure of $B$ in $\R^n$: $H$ is smooth on a open
neighborhood of $\T^n \times \bar{B}$ in $\T^n \times \R^n$ and there
exists $s_0>0$ such that, using multi-indices notation (see
Appendix~\ref{app2}),
\begin{equation}\label{defn}
  |H|_{\alpha,s_0} := c\sup_{(\theta,I) \in \T^n \times
    \bar{B}}\left(\sup_{k \in
      \N^{2n}}\frac{(|k|+1)^2{s_0}^{\alpha|k|}|\partial^k
      H(\theta,I)|}{|k|!^\alpha}\right)<\infty, \quad c:=4\pi^2/3.  
\end{equation}
This definition can be extended to vector-valued function $X: \T^n
\times \bar{B} \rightarrow \R^p$ by setting 
\begin{equation}\label{defn2}
  |X|_{\alpha,s_0} := c\sup_{(\theta,I) \in \T^n \times
    \bar{B}}\left(\sup_{k \in
      \N^{2n}}\frac{(|k|+1)^2{s_0}^{\alpha|k|}|\partial^k
      X(\theta,I)|_1}{|k|!^\alpha}\right)<\infty 
\end{equation}
where $|\,.\,|_1$ is the $l_1$-norm of vectors in $\R^p$, or the sum
of the absolute values of the components. As a rule, we will use the
$l_1$-norm for vectors, so for simplicity we shall write
$|\,.\,|_1=|\,.\,|$. To emphasize the role of the ``Gevrey width''
$s_0$, we shall also say that $H$ is $(\alpha,s_0)$-Gevrey
if~\eqref{defn} holds. Observe that a function is $1$-Gevrey if and
only it is real-analytic, in which case the parameter $s_0>0$ is the
width of analyticity.

Properties of these Gevrey norms are described in Appendix~\ref{app2};
in particular we explain there the (inessential) role of the factor
$(|k|+1)^2$ and the normalizing constant $c>0$ in~\eqref{defn}.

\subsection{The $\alpha$-Bruno-R{\"u}ssmann condition}\label{sec:aBR}

Given $\omega_0 \in \R^n$, define the function
\begin{equation}\label{eqpsi}
\Psi_{\omega_0} : [1,+\infty) \rightarrow [\Psi_{\omega_0}(1),+\infty], \quad
  Q \mapsto \max \{|k\cdot\omega_0|^{-1}\; | \; k \in \Z^n, \; 0 <
  |k|\leq Q\}.
\end{equation}
This function $\Psi_{\omega_0}$ measures the size of the so-called small denominators which will
come into play in our computations. Call $\mathrm{BR}$ the set of
vectors $\omega_0$ satisfying the so-called \emph{Bruno-R{\"u}ssmann
  condition},
\begin{equation}\label{BR}
\int_{1}^{+\infty}\frac{\ln(\Psi_{\omega_0}(Q))}{Q^2}dQ < \infty \tag{BR}
\end{equation}
and, given $\alpha \geq 1$, call $\mathrm{BR}_\alpha$ the set of
vectors $\omega_0$ satisfying the \emph{$\alpha$-Bruno-R{\"u}ssmann condition}, which we define as
\begin{equation}\label{alphaBR}
\int_{1}^{+\infty}\frac{\ln(\Psi_{\omega_0}(Q))}{Q^{1+\frac{1}{\alpha}}}dQ <
\infty. \tag{$\mathrm{BR}_\alpha$}
\end{equation}
These conditions prevent $\Psi_{\omega_0}$ from growing too fast at
infinity.  If $\omega_0 \in \mathrm{BR} = \mathrm{BR}_1$, in
particular $\Psi_{\omega_0}(Q)$ is finite for all $Q$, i.e. $\omega_0$
is non-resonant. Besides, the set $\mathrm{BR}_\alpha$ decreases with
respect to $\alpha$. For example, if
$\Psi_{\omega_0}(Q) = \exp (Q^\beta)$ then
$\omega_0 \in \mathrm{BR}_\alpha$ if and only if $\beta < 1/\alpha$
(we let the reader check, using continued fractions if $n=2$, that the
set of vectors $\omega_0$ having such function $\Psi_{\omega_0}$ is
not empty).

Let $D_\tau$ be the set of $\tau$-Diophantine vectors ($\tau\geq n-1$),
i.e. for which there exists $\gamma>0$ such that
$\Psi(Q) \leq Q^\tau/\gamma$ for all $Q\geq 1$. $D_\tau$ is non-empty and has full measure if $\tau>n-1$~\cite{Rus75}. As
definitions show, for all $\alpha\geq 1$, we have
\(D_\tau \subset \mathrm{BR}_\alpha\).
Thus, as Example~\ref{ex:BRa} shows, 
\[\cap_{\alpha\geq 1} \mathrm{BR}_\alpha \setminus \cup_{\tau \geq n-1} D_\tau\] 
has zero-measure but is non-empty.

\smallskip Now assume that $\omega_0$ is non-resonant. The function
$\Psi_{\omega_0}$ is non-decreasing, piecewise constant, and has a
countable number of discontinuities. In the sequel, it will be more
convenient to work with a continuous version of $\Psi_{\omega_0}$: it
is not hard to prove (see, for instance, Appendix A of~\cite{BF13})
that one can find a continuous non-decreasing function
$\Psi : [1,\infty) \to [\Psi(1),+\infty)$ such that
$\Psi(1)=\Psi_{\omega_0}(1)$ and
\begin{equation}\label{foncpsi}
\Psi_{\omega_0}(Q) \leq \Psi(Q) \leq \Psi_{\omega_0}(Q+1), \quad Q \geq 1.
\end{equation}
For all $k \in \Z^n\setminus\{0\}$, we still have
\[ |k\cdot \omega_0| \geq 1/\Psi(|k|) \]
and in the condition~\eqref{alphaBR} (which defines $\omega_0 \in \mathrm{BR}_\alpha$), one may use $\Psi$ instead of $\Psi_{\omega_0}$. 

Let us now define the function
$$\Delta : [1,+\infty) \to [\Psi(1),+\infty),
\quad Q \mapsto Q\Psi(Q).$$
It is continuous and increasing, and thus is a homeomorphism, whose
functional inverse is
$$\Delta^{-1} : [\Psi(1),+\infty) \to [1,+\infty), \quad \Delta^{-1} \circ
\Delta =\Delta \circ \Delta^{-1}= \mathrm{Id}.$$


In Appendix~\ref{app1} we show that the set $\mathrm{BR}_\alpha$
agrees with the set $\mathrm{A}_\alpha$ defined by the condition
\begin{equation*}
  \label{BR2}%
  \int_{\Delta (1)}^{+\infty}\frac{dx}{x(\Delta^{-1}(x))^{\frac{1}{\alpha}}} <
  \infty. \tag{$\mathrm{A}_\alpha$} 
\end{equation*}  

\section{Main results}
\label{sec:main-results}

\subsection{KAM theorem for non-degenerate integrable Hamiltonians}
\label{KAM}

The image of the map $\Theta_0 : \T^n \rightarrow \T^n \times B$,
$q \mapsto (q,0)$, is an embedded torus invariant by the flow of $h$ carrying a quasi-periodic flow with frequency $\omega_0$. We shall
prove that this quasi-periodic invariant Gevrey-smooth embedded torus
is preserved by an arbitrary small perturbation, provided $h$ is
non-degenerate, $H$ is $\alpha$-Gevrey and $\omega_0$ satisfies the
$\alpha$-Bruno-R{\"u}ssmann condition.

\begin{Main}\label{classicalKAM}
  Let $H$ be as in~\eqref{Ham1}, where $H$ is $(\alpha,s_0)$-Gevrey,
  $\omega_0\in \mathrm{BR}_\alpha$ and $h$ is non-degenerate. Then
  there exists $0<s_0'<s_0$ such that for $\epsilon$ small enough,
  there exists an $(\alpha,s_0')$-Gevrey torus embedding
  $\Theta_{\omega_0} : \T^n \rightarrow \T^n \times B$ such that
  $\Theta_{\omega_0}(\T^n)$ is invariant by the Hamiltonian flow of
  $H$ and quasi-periodic with frequency $\omega_0$.  Moreover,
$\Theta_{\omega_0}$ is close to
  $\Theta_0$ in the sense that
  \[ |\Theta_{\omega_0}-\Theta_0|_{\alpha,s_0'} \leq
  c\sqrt{\epsilon} \]
  for some constant $c>0$ independent of $\epsilon$.
\end{Main}

Theorem~\ref{classicalKAM} will be deduced from a KAM theorem for a
Hamiltonian with parameters, for which a quantitative statement is
given in~\S\ref{sec:KAMparam}. Let us also state the corresponding
iso-energetic and non-autonomous time-periodic versions.

We say that the integrable Hamiltonian $h$ is \emph{iso-energetically
  non-degenerate} at $0$ if the so-called bordered Hessian of $h$,
\[\begin{pmatrix} 
\nabla^2 h(0) & ^{t}\nabla h(0) \\
\nabla h(0) & 0 
\end{pmatrix},\]
has a non-zero determinant. Under this assumption, the unperturbed torus $p=0$, with energy $h(0)$, can be continued to a torus with the same energy but with a frequency of the form $\lambda\omega_0$ for $\lambda$ close to one. 

\begin{Main}\label{isoKAM}
  Let $H$ be as in~\eqref{Ham1}, where $H$ is $(\alpha,s_0)$-Gevrey,
  $\omega_0\in \mathrm{BR}_\alpha$ and $h$ is iso-energetically non-degenerate. Then
  there exists $0<s_0'<s_0$ such that for $\epsilon$ small enough,
  there exist $\lambda \in \R^*$ and an $(\alpha,s_0')$-Gevrey torus embedding
  $\Theta_{\omega_0} : \T^n \rightarrow \T^n \times B$ such that
  $\Theta_{\omega_0}(\T^n)$ is invariant by the Hamiltonian flow of
  $H$, contained in $H^{-1}(h(0))$ and quasi-periodic with frequency $\lambda\omega_0$.  Moreover, $\lambda$ is close to one and
$\Theta_{\omega_0}$ is close to
  $\Theta_0$ in the sense that
  \[ |\lambda-1| \leq c\sqrt{\epsilon}, \quad |\Theta_{\omega_0}-\Theta_0|_{\alpha,s_0'} \leq
  c\sqrt{\epsilon} \]
  for some constant $c>0$ independent of $\epsilon$.
\end{Main}

We can also look at the non-autonomous time-periodic version; we consider a slightly different setting by looking at a Hamiltonian function
$\tilde{H} : \T^n \times B \times \T \rightarrow \R$ of the form
\begin{equation}\label{HamT}
\begin{cases}\tag{$\tilde{*}$}
  \tilde{H}(q,p)= h(p) +\epsilon f(q,p,t), \\
  \nabla h(0):=\omega_0 \in \R^n.
\end{cases}  
\end{equation} 
It is better to consider the unperturbed torus $p=0$ as an invariant torus for the integrable Hamiltonian $\tilde{h}: B \times \R$ defined by $\tilde{h}(p,e):=h(p)+e$: it is then quasi-periodic with frequency $\tilde{\omega}_0:=(\omega_0,1)$, has dimension $n+1$ and is the image of the trivial embedding $\tilde{\Theta}_0 : \T^n \times \T \rightarrow \T^n \times B \times \T$.

\begin{Main}\label{timeKAM}
Let $\tilde{H}$ be as in~\eqref{HamT}, where $\tilde{H}$ is $(\alpha,s_0)$-Gevrey,
  $\omega_0\in \mathrm{BR}_\alpha$ and $h$ is non-degenerate. Then
  there exists $0<s_0'<s_0$ such that for $\epsilon$ small enough,
  there exists an $(\alpha,s_0')$-Gevrey torus embedding
  $\tilde{\Theta}_{\omega_0} : \T^n \times \T \rightarrow \T^n \times B \times \T$ such that
  $\tilde{\Theta}_{\omega_0}(\T^n \times \T)$ is invariant by the Hamiltonian flow of
  $\tilde{H}$ and quasi-periodic with frequency $\tilde{\omega}_0$.  Moreover,
$\tilde{\Theta}_{\omega_0}$ is close to
  $\tilde{\Theta}_0$ in the sense that
  \[ |\tilde{\Theta}_{\omega_0}-\tilde{\Theta}_0|_{\alpha,s_0'} \leq
  c\sqrt{\epsilon} \]
  for some constant $c>0$ independent of $\epsilon$.
\end{Main}

Theorem~\ref{isoKAM} and Theorem~\ref{timeKAM} are essentially equivalent statements and can be easily deduced from Theorem~\ref{classicalKAM}; in the analytic case details are given in~\cite{TreBook}, Chapter 2, but it is plain to observe that the arguments still work in the Gevrey case.

\subsection{Destruction of invariant tori}\label{sec:destruction}


According to Theorem~\ref{classicalKAM}, the $\alpha$-Bruno-R{\"u}ssmann
condition is sufficient for the preservation of an invariant torus
under an $\alpha$-Gevrey perturbation. A natural question is: is it
necessary? To this question, here we only bring a partial answer,
which circumscribes the optimal arithmetic condition, if any.
Following Bessi~\cite{Bes00}, one can show that if $\omega=\omega_0$
satisfies a condition (the condition~\eqref{conddestruction} defined
below), the torus can be destroyed. In particular, this shows that the
exponent $1+1/\alpha$ in~\eqref{alphaBR} cannot be replaced by a
strictly larger exponent. As a matter of fact, the example of Bessi
already shows this in the analytic case $\alpha=1$; our observation
here is that Bessi's example gives a similar result for any
$\alpha \geq 1$.

\begin{Main}\label{destruction}
  Given $\alpha \geq 1$, assume that the vector $\omega \in \R^n$
  satisfies the following condition:
\begin{equation}\label{conddestruction}
\limsup_{Q \rightarrow +\infty}
\frac{\ln(\Psi_\omega(Q))}{Q^{1/\alpha}}>0.\tag{$\mathrm{B}_\alpha$} 
\end{equation}
Then an invariant torus with frequency $\omega$ can be destroyed by an
arbitrarily small $\alpha$-Gevrey perturbation.
\end{Main}

Thus the condition that $\omega_0$ does not
satisfy~\eqref{conddestruction}, namely
\begin{equation}\label{coho}
\lim_{Q \rightarrow +\infty}
\frac{\ln(\Psi_\omega(Q))}{Q^{1/\alpha}}=0, \tag{$\mathrm{R}_\alpha$}
\end{equation}
is a necessary condition for the conclusion of
Theorem~\ref{classicalKAM} to hold true. For $\alpha=1$, this
condition~\eqref{coho} is actually a sufficient (and most probably
necessary) condition to solve the cohomological equation associated to
$\omega$ (see~\cite{Rus75}); in the general case $\alpha \geq 1$ this
should also be true but we couldn't find a reference. Let us also note
that~\eqref{coho} is implied by (but clearly not equivalent to) the
condition that $\omega \in \mathrm{BR}_\alpha$, see Remark~\ref{rem}
in Appendix~\ref{app1}.

For a more precise statement and how this follows from~\cite{Bes00},
we refer to Theorem~\ref{BessiG} in Section~\ref{sec:bessi}. It is
likely that one could improve this result for $\alpha>1$ by using
perturbations with compact support as in \cite{CW13}.

Observe that for any $\alpha \geq 1$ and any $0 < \beta < \alpha$,
vectors $\omega \in \R^n$ for which
\[ \Psi_\omega(Q) \sim e^{Q^{1/\alpha}} \]
satisfies~\eqref{conddestruction} but also the $\beta$-Bruno-R{\"u}ssmann
condition. (That such vectors do exist is a classical matter in number
theory.) The following corollary is then obvious.

\begin{corollary}\label{cordestruction}
For any $\alpha \geq 1$ and any $0 < \beta < \alpha$, there exist invariant tori with frequency vectors $\omega \in \mathrm{BR}_\beta$ which can be destroyed by an arbitrary small
$\alpha$-Gevrey perturbation. In particular, there exist invariant tori with frequency vectors $\omega \in \mathrm{BR}$ which can be destroyed by an arbitrary small Gevrey non-analytic perturbation.
\end{corollary}

\subsection{KAM theorem for constant vector fields on the torus}

Now we state a Gevrey version of Arnold's normal form theorem for vector fields on the torus.

\begin{Main}\label{KAMvector}
Let $\omega_0 \in \mathrm{BR}_\alpha$ and $X \in G_{\alpha,s}(\T^n,\R^n)$ a vector field on $\T^n$ of the
  form
  \[X = \omega_0 + B, \quad |B|_{\alpha,s} \leq \mu.\]
  Then, for $\mu$ sufficiently small, there exist a vector
  $\omega_0^* \in \R^n$ and an $(\alpha,s/2)$-Gevrey diffeomorphism
  $\Xi : \T^n \to \T^n$ such that $X+\omega_0^*-\omega_0$ is conjugate
  to $\omega_0$ via $\Xi$:
\[ \Xi^*(X+\omega_0^*-\omega_0)=\omega_0.\] 
Moreover, we have the estimate
\[ |\omega_0^*-\omega_0| \leq c\mu, \quad |\Xi-\mathrm{Id}|_{\alpha,s/2} \leq c\mu \]
for some constant $c\geq 1$ independent of $\mu$.
\end{Main}

Observe that because of the shift of frequency $\omega_0^*-\omega_0$,
in general this result does not give any information on the vector
field $X$. Under some further assumption (for instance, if $\omega_0$
belongs to the rotation set of $X$, see~\cite{Kar16}), then this shift
vanishes and Theorem~\ref{KAMvector} implies that $X$ is conjugated to
$\omega_0$.

An even more restricted setting is when $X$ is proportional to
$\omega_0$ (so that the flow of $X$ is a re-parametrization of the
linear flow of frequency $\omega_0$ and thus $\omega_0$ is the unique
rotation vector of $X$); Theorem~\ref{KAMvector} applies in this case
to give a conjugacy to $\omega_0$, assuming that
$\omega_0 \in \mathrm{BR}_\alpha$, but the proof is actually much
simpler in this case (it boils down to solve only once the
cohomological equation) and should require the weaker condition that
$\omega_0$ satisfies~\eqref{coho}, as it is stated in the case
$\alpha=1$ in~\cite{Fay02}. Still in~\cite{Fay02}, it is proved that
for $\alpha=1$ (there are also versions in the $C^r$ case), if
$\omega_0$ satisfies~\eqref{conddestruction}, then there is a dense
set of reparametrized linear flow which are weak-mixing (and so cannot
be conjugated to the linear flow); thus a necessary condition for
Theorem~\ref{KAMvector} to hold true is that $\omega_0$
satisfies~\eqref{coho} (and this is also a sufficient condition if we
impose that $X$ is proportional to $\omega_0$)\footnote{We would like
  to thank B. Fayad for a discussion on this topic.}. Clearly, this
should extend to the general case $\alpha \geq 1$ and thus the
condition that $\omega_0$ does not satisfy~\eqref{conddestruction} is
a necessary condition for Theorem~\ref{KAMvector} to hold true, as in
Theorem~\ref{classicalKAM}.

\subsection{KAM theorem for maps}

In this section, we give the statement of discrete versions of
Theorem~\ref{classicalKAM} and Theorem~\ref{KAMvector}.

Let us start with the discrete analogue of
Theorem~\ref{classicalKAM}. Given a function
$h : \bar{B} \rightarrow \R$, we define the exact-symplectic map
\[ F_h : \T^n \times \bar{B} \rightarrow \T^n \times \bar{B}, \quad
(q,p) \mapsto (q+\nabla h (p),p).  \]
As before, let us fix $\alpha \geq 1$ and $s_0>0$.

\begin{Main}\label{KAMmap1}
  Let $F : \T^n \times \bar{B} \rightarrow \T^n \times \bar{B}$ be an
  $(\alpha,s_0)$-Gevrey exact symplectic map with
  \[ |F-F_h|_{\alpha,s_0} \leq \epsilon. \]
  Assume that $\omega_0=\nabla h(0) \in \mathrm{BR}_\alpha$ and that
  $h$ is non-degenerate. Then there exists $0<s_0'<s_0$ such that for
  $\epsilon$ small enough, there exists an $(\alpha,s_0')$-Gevrey
  torus embedding $\Theta_{\omega_0} : \T^n \rightarrow \T^n \times B$
  such that $\Theta_{\omega_0}(\T^n)$ is invariant by $F$ and
  $\Theta_{\omega_0}$ gives a conjugacy between the translation of
  vector $\omega_0$ on $\T^n$ and the restriction of $F$ to
  $\Theta_{\omega_0}(\T^n)$.  Moreover, $\Theta_{\omega_0}$ is close
  to $\Theta_0$ in the sense that
  \[ |\Theta_{\omega_0}-\Theta_0|_{\alpha,s_0'} \leq
  c\sqrt{\epsilon} \]
  for some constant $c>0$ independent of $\epsilon$.
\end{Main}

Theorem~\ref{KAMmap1} follows at once from Theorem~\ref{isoKAM} (or Theorem~\ref{timeKAM}) provided one has a suitable quantitative ``suspension" result; in the analytic case $\alpha=1$ this was proved in~\cite{KP94} and in the Gevrey non-analytic case $\alpha>1$ this is contained in~\cite{LMS16}.

In the same way, we have the following discrete analogue of Theorem~\ref{KAMvector}. Given $\omega_0 \in \R^n$, let $T_{\omega_0}$ be the translation of $\T^n$ of vector $\omega_0$:
\[ T_{\omega_0} : \T^n \rightarrow \T^n, \quad \theta \mapsto \theta+\omega_0. \] 
Let $\alpha \geq 1$ and $s>0$.

\begin{Main}\label{KAMmap2}
Let $\omega_0 \in \mathrm{BR}_\alpha$ and $T \in G_{\alpha,s}(\T^n,\T^n)$ a diffeomorphism of $\T^n$ of the form
\[T = T_{\omega_0} + B, \quad |B|_{\alpha,s} \leq \mu.\]
Then, for $\mu$ sufficiently small, there exist a vector $\omega_0^* \in \R^n$ and an $(\alpha,s/2)$-Gevrey diffeomorphism
  $\Xi : \T^n \to \T^n$ such that $T+\omega_0^*-\omega_0$ is conjugate to $T_{\omega_0}$ via $\Xi$:
\[ \Xi^{-1} \circ (T+\omega_0^*-\omega_0) \circ \Xi=T_{\omega_0}.\] 
Moreover, we have the estimate
\[ |\omega_0^*-\omega_0| \leq c\mu, \quad |\Xi-\mathrm{Id}|_{\alpha,s/2} \leq c\mu \]
for some constant $c\geq 1$ independent of $\mu$.
\end{Main}

\section{Statement of the KAM theorem with parameters}
\label{sec:KAMparam}

Let us now consider the following setting. Fix
$\omega_0 \in \R^n \setminus \{0\}$. Re-ordering the components of $\omega_0$ and re-scaling the Hamiltonian allow us to assume without loss of generality that
\[\omega_0=(1,\bar{\omega}_0) \in \R^n, \quad \bar{\omega}_0 \in
[-1,1]^{n-1}.\]
Given real numbers $r>0$ and $h>0$, we let
\[ D_r:=\{I \in \R^n \; | \; |I| \leq r \}, \quad D_h:=\{\omega \in
\R^n \; | \; |\omega-\omega_0| \leq h \}, \quad D_{r,h}:=D_r \times
D_h. \]
Our Hamiltonians will be defined on $\T^n \times D_{r,h}$, a
neighborhood of $\T^n \times \{0\} \times \{\omega_0\}$ in
$\T^n \times \R^n \times \R^n$.

Let $\alpha \geq 1$, $s>0$, $\eta \geq 0$ a fixed parameter,
$\varepsilon \geq 0$ and $\mu \geq 0$ two small parameters. We
consider a function $H \in G_{\alpha,s}(\T^n \times D_{r,h})$ of the
form
\begin{equation}\label{HAM}
\begin{cases}
  H(\theta,I,\omega) = \underbrace{e(\omega)+\omega \cdot
    I}_{N(I,\omega)} +
  \underbrace{A(\theta,\omega)+B(\theta,\omega)\cdot
    I}_{P(\theta,I,\omega)}+ 
  \underbrace{M(\theta,I,\omega) \cdot I^2}_{R(\theta,I,\omega)} \\
|A|_{\alpha,s}\leq \varepsilon, \quad |B|_{\alpha,s} \leq \mu, \quad
  |\nabla_I^2 R|_{\alpha,s} \leq \eta
\end{cases}     
  \tag{$**$}
\end{equation}
where the notation $M(\theta,I,\omega) \cdot I^2$ stands for the vector $I$
given twice as an argument to the symmetric bilinear form
$M(\theta,I,\omega)$. Observe that
$A : \T^n \times D_h \rightarrow \R$,
$B : \T^n \times D_h \rightarrow \R^n$ whereas
$M : \T^n \times D_{r,h} \rightarrow M_n(\R)$ with
$M_n(\R)$ the ring of real square matrices of size $n$. Observe that we do not assume $\varepsilon = \mu$ because these two small parameters play different roles in applications (in
Theorem~\ref{classicalKAM} we will have $\mu = \sqrt{\varepsilon}$ while in
Theorem~\ref{KAMvector}, $\varepsilon=0$ and $\mu$ will be the only small parameter).

The function $H$ in~\eqref{HAM} should be considered as a Gevrey
Hamiltonian on $\T^n \times D_r$, depending on a parameter
$\omega \in D_h$; for a fixed parameter $\omega \in D_h$, when
convenient, we will write
\[ H_\omega (I,\theta)=H(I,\theta,\omega), \quad
N_\omega(I)=N(I,\omega), \quad P_\omega(I,\theta)=P(I,\theta,\omega),
\quad R_\omega(I,\theta)=R(I,\theta,\omega). \]
The image of the map $\Phi_0 : \T^n \rightarrow \T^n \times D_r$,
$\theta \mapsto (\theta,0)$ is a smooth embedded torus in
$\T^n \times D_r$, invariant by the Hamiltonian flow of
$N_{\omega_0}+R_{\omega_0}$ and quasi-periodic with frequency
$\omega_0$. The next theorem asserts that this quasi-periodic torus
will persist, being only slightly deformed, as an invariant torus not
for the Hamiltonian flow of $H_{\omega_0}$ but for the Hamiltonian
flow of $H_{\omega_0^*}$, where $\omega_0^*$ is a parameter close to
$\omega_0$, provided $\varepsilon$ and $\mu$ are sufficiently small
and $\omega_0$ satisfies the $\alpha$-Bruno-R{\"u}ssmann condition. Here
is the precise statement.

\begin{Main}\label{KAMparameter}
  Let $H$ be as in~\eqref{HAM}, with $\omega_0 \in
  \mathrm{BR}_\alpha$.
  Then there exist positive constants $c_1 \leq 1$, $c_2\leq 1$ and
  $c_3 \geq 1$ depending only on $n$ and $\alpha$ such that if
  \begin{equation}\label{cond0}
    \sqrt{\varepsilon} \leq \mu \leq h/2, \quad \sqrt{\varepsilon} \leq r, \quad h \leq c_1 (Q_0\Psi(Q_0))^{-1}
  \end{equation}
  where $Q_0 \geq n+2$ is sufficiently large so that
  \begin{equation}\label{eqQ0}
    Q_0^{-\frac{1}{\alpha}}+(\ln 2)^{-1}\int_{\Delta(Q_0)}^{+\infty}\frac{dx}{x(\Delta^{-1}(x))^{\frac{1}{\alpha}}} \leq c_2(1+\eta)^{-1/\alpha} s,
  \end{equation}
  there exist a vector $\omega_0^* \in \R^n$ and an
  $(\alpha,s/2)$-Gevrey embedding
  \[ \Phi^*_{\omega_0} : \T^n \times D_{r/2} \rightarrow \T^n \times D_r \] 
  of the form
  \[ \Phi^*_{\omega_0}(\theta,I)=(\theta+E^*(\theta),I+F^*(\theta)\cdot I+G^*(\theta)) \]
  with the estimates
  \begin{equation}\label{estim0}
    |\omega_0^*-\omega_0| \leq c_3 \mu, \quad |E^*|_{\alpha,s/2} \leq
    c_3 \Psi(Q_0)\mu, \quad |F^*|_{\alpha,s/2}\leq c_3
    \Delta(Q_0)\mu,   \quad |G^*|_{\alpha,s/2}\leq c_3
    \Delta(Q_0)\varepsilon   
  \end{equation}
  and  such that 
  \[ H_{\omega_0^*} \circ \Phi_{\omega_0}^*(\theta,I)=e_0^*+\omega_0 \cdot I
  + R^*(\theta,I), \quad R^*(\theta,I) = M^*(\theta,I) \cdot I^2,\]
  with the estimates
  \[ |e_0^*-e_{\omega_0^*}| \leq c_3\varepsilon, \quad |\nabla_I^2 R^*-
  \nabla_I^2 R_{\omega_0^*}|_{\alpha,s/2} \leq c_3\eta \Delta(Q_0)\mu.\]
\end{Main}

Theorem~\ref{classicalKAM} follows quite
directly from Theorem~\ref{KAMparameter}, introducing the frequencies
$\omega = \nabla h (p)$ as independent parameters, taking
$\mu=\sqrt{\varepsilon}$, and tuning the shift of frequency
$\omega_0^*-\omega_0$ using the non-degeneracy assumption on the
unperturbed Hamiltonian. Theorems~\ref{KAMvector} follows also from Theorem~\ref{KAMparameter} by realizing $X$ as the restriction of a Hamiltonian vector field on an invariant torus,  setting $\varepsilon=\eta=0$ and letting $\mu$ be the only small parameter. These arguments are made precise in
Section~\ref{sec:proofs}. 

\section{Proofs of Theorems~\ref{classicalKAM} and \ref{KAMvector}, assuming Theorem~\ref{KAMparameter}}
\label{sec:proofs}

\subsection{Proof of Theorem~\ref{classicalKAM}}
\label{sec:proof-classical}

In this section, we assume Theorem~\ref{KAMparameter} and we show how
it implies Theorem~\ref{classicalKAM}, following~\cite{Pos01} (in the
analytic case) and~\cite{Pop04} (in the Gevrey case).

\begin{proof}[Proof of Theorem~\ref{classicalKAM}]
  As noticed at the beginning of Section~\ref{sec:KAMparam}, we may
  assume that $\omega_0$ is of the form
  \[\omega_0=(1,\bar{\omega}_0) \in \R^n, \quad \bar{\omega}_0 \in
  [-1,1]^{n-1}.\]
For $p_0 \in B$, we expand $h$ in a small neighborhood of $p_0$: writing $p=p_0+I$ for $I$ close to zero, we get
\[ h(p)=h(p_0)+ \nabla_p h(p_0)\cdot I + \int_{0}^{1}(1-t)\nabla_p^2
h(p_0 +tI)\cdot I^2 \, dt. \]
Similarly, we expand $\epsilon f$ with respect to $p$, in a small neighborhood of $p_0$:
\[ \epsilon f(q,p)=\epsilon f(q,p_0)+ \epsilon \nabla_p f(q,p_0)\cdot
I + \epsilon\int_{0}^{1}(1-t)\nabla_p^2 f(q,p_0 +tI) \cdot I^2\, dt. \]
Since $\nabla_p h : B \rightarrow \Omega$ is a diffeomorphism, instead
of $p_0$ we can use $\omega=\nabla_p h(p_0)$ as a new variable, and
letting $\nabla_\omega g:=(\nabla h)^{-1}$, we write
\[ h(p)=e(\omega)+ \omega\cdot I + R_h(I,\omega) \] 
with 
\[ e(\omega):=h(\nabla_\omega g(\omega)), \quad
R_h(I,\omega):=\int_{0}^{1}(1-t)\nabla_p^2 h(\nabla_\omega
g(\omega)+tI) \cdot I^2 \, dt\]
and also, letting $\theta=q$, 
\[ \epsilon f(q,p)=\epsilon \tilde{A}(\theta,\omega)+ \epsilon \tilde{B}(\theta,\omega)\cdot I + \epsilon R_f(\theta,I,\omega) \]
with
\[ \tilde{A}(\theta,\omega):=f(\theta,\nabla_\omega g(\omega)), \quad \tilde{B}(\theta,\omega):=\nabla_p f(\theta,\nabla_\omega g(\omega))\] 
and
\[ R_f(\theta,I,\omega):=\epsilon\int_{0}^{1}(1-t)\nabla_p^2
f(\theta,\nabla_\omega g(\omega) +tI) \cdot I^2\, dt.   \]
Finally, we can set
\[ A:=\epsilon \tilde{A}, \quad B:=\epsilon \tilde{B}, \quad
R:=R_h+\epsilon R_f  = M(\theta,I,\omega)\cdot I^2, \] 
so that $h+\epsilon f$ can be written as
\[ H(\theta,I,\omega)=e(\omega)+\omega \cdot
I+A(\theta,\omega)+B(\theta,\omega)\cdot I+R(\theta,I,\omega),\]
and we have
\[  
\nabla_I^2 R(\theta,I,\omega)
=\nabla_I^2 h(\nabla_\omega
  g(\omega)+I)+\epsilon\nabla_I^2 f(\theta,\nabla_\omega g(\omega)+I).
\]
By assumption, $h$ and $f$ are $(\alpha,s_0)$-Gevrey on
$\T^n \times \bar{B}$, and since the space of Gevrey functions is
closed under taking derivatives, products, composition and inversion
(up to restricting the parameter $s_0$, see Appendix~\ref{app2} for
the relevant estimates), we claim that we can find $s>0$, $r>0$,
$h>0$ and $\tilde{c}>0$ which are independent of $\epsilon$ such that $H$ is $(\alpha,s)$-Gevrey on the domain $\T^n \times D_{r,h}$ with the estimates
\[ |A|_{\alpha,s} \leq \tilde{c}\epsilon, \quad |B|_{\alpha,s} \leq
\tilde{c}\epsilon,\quad
|\nabla_I^2 R|_{\alpha,s} \leq \tilde{c}.  \] 
We may set
\[ \varepsilon:=\tilde{c}\epsilon, \quad \mu:=\sqrt{\varepsilon}, \quad \eta:=\tilde{c} \]
and assuming $\epsilon$ small enough, we have $\tilde{c}\epsilon \leq \mu=\sqrt{\varepsilon}$. Thus we have
\[ |A|_{\alpha,s} \leq \varepsilon, \quad |B|_{\alpha,s} \leq \mu, \quad |\nabla_I^2 R|_{\alpha,s} \leq \eta.    \]
Having fixed $s>0$ and $r>0$, we may choose $Q_0$ sufficiently large so that~\eqref{eqQ0} holds true, and then by further restricting first $h$, and then $\epsilon$, we may assume that the condition~\eqref{cond0} is satisfied. Theorem~\ref{KAMparameter} applies: there exist an $(\alpha,s/2)$-Gevrey embedding $\Upsilon_{\omega_0} : \T^n \rightarrow \T^n \times D_r$, defined by 
\[ \Upsilon_{\omega_0}(\theta):=\Phi_{\omega_0}(\theta,0)=(\theta+E^*(\theta),G^*(\theta))\] 
where $\Phi_{\omega_0}$ is given by Theorem~\ref{KAMparameter}, and a vector $\omega_0^* \in \R^n$ such that $\Upsilon_{\omega_0}(\T^n)$ is invariant by the Hamiltonian flow of $H_{\omega_0^*}$ and quasi-periodic with frequency $\omega_0$. Moreover, $\omega_0^*$ and $\Upsilon_{\omega_0}$ satisfy the estimates
\begin{equation*}
|\omega_0^*-\omega_0| \leq c \mu, \quad |E^*|_{\alpha,s/2} \leq c \Psi(Q_0)\mu, \quad |G^*|_{\alpha,s/2}\leq c Q_0\Psi(Q_0)\varepsilon 
\end{equation*}
for some large constant $c>1$. Since $h$ is non-degenerate, there exists $p_0^*$ such that $\nabla h (p_0^*)=\omega_0^*$ and, up to taking $c>1$ larger and recalling that $\mu=\sqrt{\varepsilon}$, the above estimates imply 
\begin{equation}\label{estFFF}
|p_0^*| \leq c \sqrt{\varepsilon}, \quad |E^*|_{\alpha,s/2} \leq c \Psi(Q_0)\sqrt{\varepsilon}, \quad |G^*|_{\alpha,s/2}\leq c Q_0\Psi(Q_0)\varepsilon.  
\end{equation}
Now observe that an orbit $(\theta(t),I(t))$ for the Hamiltonian $H_{\omega_0^*}$ corresponds to an orbit $(q(t),p(t))=(\theta(t),I(t)+p_0^*)$ for our original Hamiltonian. Hence, if we define $T : \T^n \times \R^n \rightarrow \T^n \times \R^n$ by $T(\theta,I)=(\theta,I+p_0^*)$ and 
\[ \Theta_{\omega_0} = T \circ \Upsilon_{\omega_0} : \T^n \rightarrow \T^n
\times \R^n, \quad
\Theta_{\omega_0}(\theta)=(\theta+E^*(\theta),G^*(\theta)+p_0^*) \]
then $\Theta_{\omega_0}$ is an $(\alpha,s/2)$-Gevrey torus embedding
such that $\Theta_{\omega_0}(\T^n)$ is invariant by the Hamiltonian
flow of $H$ and quasi-periodic with frequency $\omega_0$.  The
estimates on the distance between $\Theta_{\omega_0}$ and the trivial
embedding $\Theta_0$ follows directly from~\eqref{estFFF}, which
finishes the proof.
\end{proof}

\subsection{Proof of Theorem~\ref{KAMvector}}

Now we show how Theorem~\ref{KAMvector} follows from Theorem~\ref{KAMparameter}.

\begin{proof}[Proof of Theorem~\ref{classicalKAM}]
Consider the vector field $X=\omega_0+B \in G_{\alpha,s}(\T^n,\R^n)$ as in the statement. It can be trivially included into a parameter-depending vector field: given $h>0$, let $\hat{X} \in G_{\alpha,s}(\T^n \times D_h,\R^n)$ be such that  
\[ \hat{X}(\theta,\omega)=\hat{X}_\omega(\theta)=\omega+B(\theta), \quad \omega \in D_h, \quad \hat{X}_{\omega_0}=X.  \]
Now given any $r>0$, consider the Hamiltonian $H$ defined on $\T^n \times D_{r,h}$ by
\begin{equation}\label{Hamvector}
H(\theta,I,\omega)=H_\omega(\theta,I):=\omega\cdot I +B(\theta)\cdot I. 
\end{equation}
Clearly, for any parameter $\omega$, the torus $\T^n \times \{0\}$ is invariant by the Hamiltonian vector field $X_{H_\omega}$, and, upon identifying $\T^n \times \{0\}$ with $\T^n$, the restriction of $X_{H_\omega}$ to this torus coincides with $\hat{X}_\omega$.  

Now the Hamiltonian $H$ defined in~\eqref{Hamvector} is of the form~\eqref{HAM} with $\varepsilon=\eta=0$ (and $e=0$) and therefore for $\mu$ sufficiently small, Theorem~\ref{KAMparameter} applies: there exist a vector $\omega_0^* \in \R^n$ and an $(\alpha,s/2)$-Gevrey embedding
\[ \Phi^*_{\omega_0} : \T^n \times D_{r/2} \rightarrow \T^n \times D_r \] 
here of the form
\[ \Phi^*_{\omega_0}(\theta,I)=(\theta+E^*(\theta),I+F^*(\theta)\cdot I) \]
with the estimates
  \begin{equation*}
    |\omega_0^*-\omega_0| \leq c_3 \mu, \quad |E^*|_{\alpha,s/2} \leq
    c_3 \Psi(Q_0)\mu, \quad |F^*|_{\alpha,s/2}\leq c_3
    \Delta(Q_0)\mu   
  \end{equation*}
  and  such that 
\begin{equation}\label{relH}
H_{\omega_0^*} \circ \Phi_{\omega_0}^*(\theta,I)=\omega_0 \cdot I.
\end{equation}
The embedding $\Phi^*_{\omega_0}$ clearly leaves invariant the torus $\T^n \times \{0\}$ and induces a diffeomorphism of this torus that can be identified to $\Xi:=\mathrm{Id}+E^*$. Writing the equality~\eqref{relH} in terms of Hamiltonian vector fields, we have, upon restriction to the invariant torus and recalling that the restriction of $X_{H_\omega}$ coincides with $\hat{X}_\omega$, 
\[ \Xi^*(\hat{X}_{\omega_0^*})=\omega_0. \]    
But $\hat{X}_{\omega_0^*}=\hat{X}_{\omega_0}+\omega_0^*-\omega_0=X+\omega_0^*-\omega_0$
and therefore
\[ \Xi^*(X+\omega_0^*-\omega_0)=\omega_0 \]
which, together with the estimates on $\omega_0^*$ and $\Xi-\mathrm{Id}=E^*$, was the statement we wanted to prove.    
\end{proof}

\section{Proof of Theorem~\ref{KAMparameter}}
\label{sec:proof}

This section is devoted to the proof of Theorem~\ref{KAMparameter}, in
which we will construct, by an iterative procedure, a vector
$\omega_0^*$ close to $\omega_0$ and a Gevrey-smooth torus embedding
$\Phi_{\omega_0}^*$ whose image is invariant by the Hamiltonian flow of
$H_{\omega_0^*}$.  We start, in \S\ref{sec:approx}, by recalling the
Diophantine result of~\cite{BF13} which will be used in our
approach. Then, in \S\ref{sec:step}, we describe an elementary step of
our iterative procedure, and finally, in \S\ref{sec:iteration}, we
will show that infinitely many steps may be carried out, to converge
towards a solution.

In this paper, we do not pay attention to how constants depend on the
dimension $n$ or the Gevrey-exponent $\alpha$, both being fixed. Hence
in this section, we shall write
$$u \MP v \quad (\mbox{respectively } u \PM v)$$
if, for some constant $C\geq 1$ depending only on $n$ and $\alpha$, we
have $u\leq Cv$ (respectively $Cu \leq v$). In particular, $u \PM v$ is
stronger than $u \MP v$.

\subsection{Approximation by rational vectors}\label{sec:approx}

Recall that we have written $\omega_0=(1,\bar{\omega}_0) \in \R^n$
with $\bar{\omega}_0 \in [-1,1]^{n-1}$. For a given $Q\geq 1$, it is
always possible to find a rational vector $v=(1,p/q) \in \Q^n$, with
$p \in \Z^{n-1}$ and $q \in \N$, which is a $Q$-approximation in the
sense that $|q\omega_0 - qv|\leq Q^{-1}$, and for which the
denominator $q$ satisfies the upper bound $q \leq Q^{n-1}$: this is
essentially the content of Dirichlet's theorem on simultaneous
rational approximations, and it holds true without any assumption on
$\omega_0$. In our situation, since we have assumed that $\omega_0$ is
non-resonant, there exist not only one, but $n$ linearly independent
rational vectors in $\Q^n$ which are $Q$-approximations. Moreover, one
can obtain not only linearly independent vectors, but rational vectors
$v_1,\dots,v_n$ of denominators $q_1, \dots,q_n$ such that the
associated integer vectors $q_1v_1,\dots,q_nv_n$ form a $\Z$-basis of
$\Z^n$. However, the upper bound on the corresponding denominators
$q_1, \dots, q_n$ is necessarily larger than $Q^{n-1}$, and is given
by a function of $Q$ that we can call here $\Psi_{\omega_0}'$
(see \cite{BF13} for more precise and general information, but note
that in this reference, $\Psi_{\omega_0}'$ was denoted by $\Psi_{\omega_0}$ and $\Psi_{\omega_0}$,
which we defined in~\eqref{eqpsi}, was denoted by $\Psi_{\omega_0}'$). A consequence of the main Diophantine result of \cite{BF13} is that this function $\Psi_{\omega_0}'$ is in fact essentially equivalent to the function $\Psi_{\omega_0}$.  

\begin{proposition}\label{dio}
Let $\omega_0=(1,\bar{\omega}_0) \in \R^n$ be a non-resonant vector, with $\bar{\omega}_0 \in [-1,1]^{n-1}$. For any $Q\geq n+2$, there exist $n$ rational vectors $v_1, \dots, v_n$, of denominators $q_1, \dots, q_n$, such that $q_1v_1, \dots, q_nv_n$ form a $\Z$-basis of $\Z^n$ and for $j\in\{1,\dots,n\}$,
\[ |\omega_0-v_j|\MP(q_j Q)^{-1}, \quad 1 \leq q_j \MP \Psi(Q).\]
\end{proposition}

For a proof of the above proposition with $\Psi_{\omega_0}$ instead of $\Psi$, we refer to \cite{BF13}, Theorem 2.1 and Proposition $2.3$; now by~\eqref{foncpsi}, $\Psi_{\omega_0} \leq \Psi$ and so one may replace $\Psi_{\omega_0}$ by $\Psi$. 

Now given a $q$-rational vector $v$ and a smooth function $H$ defined on $\T^n \times D_{r,h}$, we define
\begin{equation}\label{moyv}
[H]_v(\theta,I,\omega)= \int^{1}_{0} H(\theta+tqv,I,\omega)dt.
\end{equation}
Given $n$ rational vectors $v_1,\dots,v_n$, we let $[H]_{v_1,\dots,v_d}=[\cdots[H]_{v_1}\cdots]_{v_d}$. Finally we define
\begin{equation}\label{moy}
[H](I,\omega)=\int_{\T^n}H(\theta,I,\omega)d\theta.
\end{equation}
The following proposition is a consequence of the fact that the vectors $q_1v_1, \dots, q_nv_n$ form a $\Z$-basis of $\Z^n$. 

\begin{proposition}[{\cite[Corollary 6]{Bou13}}]\label{cordio}
Let $v_1,\dots,v_n$ be rational vectors, of denominators $q_1,\dots,q_n$, such that $q_1v_1, \dots, q_nv_n$ form a $\Z$-basis of $\Z^n$, and $H$ a function defined on $\T^n \times D_{r,h}$. Then
\[ [H]_{v_1,\dots,v_n}=[H]. \]  
\end{proposition} 

\subsection{KAM step}\label{sec:step}

Now we describe an elementary step of our iterative procedure. Such a
step consists in pulling back the Hamiltonian $H$ by a transformation
of the form
\[ \mathcal{F}=(\Phi,\varphi): (\theta,I,\omega) \mapsto
(\Phi(\theta,I,\omega),\varphi(\omega));\]
$\Phi$ is a parameter-depending change of coordinates and $\varphi$ a
change of parameters. Moreover, our change of coordinates will be of
the form
\[\Phi(\theta,I,\omega)=
\Phi_\omega(\theta,I)=(\theta+E(\theta,\omega),I+F(\theta,\omega)\cdot
I+G(\theta,\omega)) \] with
\[ E : \T^n \times D_h \rightarrow \R^n, \quad F : \T^n \times D_h
\rightarrow M_n(\R), \quad G : \T^n \times D_h \rightarrow \R^n \]
and for each fixed parameter $\omega$, $\Phi_\omega$ will be
symplectic. For simplicity, we shall write $\Phi=(E,F,G)$; the
composition of such transformations
$\mathcal{F}=(\Phi,\varphi)=(E,F,G,\varphi)$ is again a transformation
of the same form, and we shall denote by $\mathcal{G}$ the groupoid of
such transformations.

\begin{proposition}\label{kamstep}
  Let $H$ be as in~\eqref{HAM}, with
  $\omega_0=(1,\bar{\omega}_0) \in \R^n$ non-resonant, consider
  $0<\sigma< s$, $0<\delta<r$, $Q \geq n+2$, and assume that
  \begin{equation}\label{cond1}
    \sqrt{\varepsilon} \leq \mu \leq h/2, \quad
    \sqrt{\varepsilon} \leq r, \quad h \PM (Q\Psi(Q))^{-1}, \quad r\mu \PM \delta(Q\Psi(Q))^{-1},
    \quad (1+\eta) \PM Q\sigma^{\alpha}. 
  \end{equation}
  Then there exists an $(\alpha,s-\sigma)$-Gevrey symplectic
  transformation
  \[ \mathcal{F}=(\Phi,\varphi)=(E,F,G,\varphi) : \T^n \times
  D_{r-\delta,h/2} \rightarrow \T^n \times D_{r,h} \in \mathcal{G}, \]
  with the estimates
  \begin{equation}\label{stepest1}
    \begin{cases}
|E|_{\alpha,s-\sigma} \MP \Psi(Q)\mu, \quad |\nabla E|_{\alpha,s-\sigma} \MP \sigma^{-\alpha}\Psi(Q)\mu, \\
|F|_{\alpha,s-\sigma} \MP \sigma^{-\alpha}\Psi(Q)\mu, \quad |\nabla F|_{\alpha,s-\sigma} \MP \sigma^{-2\alpha}\Psi(Q)\mu, \\
|G|_{\alpha,s-\sigma} \MP \sigma^{-\alpha}\Psi(Q)\varepsilon, \quad |\nabla G|_{\alpha,s-\sigma} \MP \sigma^{-2\alpha}\Psi(Q)\varepsilon,   \\
|\varphi-\mathrm{Id}|_{\alpha,s-\sigma} \leq \mu, \quad |\nabla \varphi-\mathrm{Id}|_{\alpha,s-\sigma} \MP \sigma^{-\alpha} \mu
\end{cases}
  \end{equation}
  such that 
  \[H \circ \mathcal{F}(\theta,I,\omega) =
  \underbrace{e^+(\omega)+\omega\cdot I }_{N^+(I,\omega)} +
  \underbrace{A^+(\theta)+B^+(\theta)\cdot I}_{P^+(\theta,I,\omega)} +
  \underbrace{M^+(\theta,I,\omega)\cdot I^2}_{R^+(\theta,I,\omega)}, 
  \]
  with the estimates
  \begin{equation}\label{stepest2}
    \begin{cases}
      |A^+|_{\alpha,s-\sigma} \leq \varepsilon/16, \quad
      |B^+|_{\alpha,s-\sigma} \leq \mu/4, \\
      |e^+-e \circ \varphi|_{\alpha,s-\sigma} \leq |A|_{\alpha,s},
      \quad |\nabla_I^2 R^+- \nabla_I^2 R \circ
      \mathcal{F}|_{\alpha,s-\sigma}\MP \eta |F|_{\alpha,s-\sigma}.
    \end{cases}
  \end{equation}
\end{proposition}

\begin{proof}
  We divide the proof of the KAM step into five small steps. Except
  for the last one, the parameter $\omega \in D_h$ will be fixed, so
  for simplicity, in the first four steps we will drop the dependence
  on the parameter $\omega \in D_h$. Let us first notice
  that~\eqref{cond1} clearly implies the following seven inequalities:
  \begin{align}
    &h \PM (Q\Psi(Q))^{-1} \label{eq:cond1}\\ 
    &\Psi(Q)\mu \PM \sigma^{\alpha} \label{eq:cond2}\\ 
    &\varepsilon \leq r\mu \label{eq:cond3}\\
    &r\mu\sigma^{-\alpha}\Psi(Q) \PM \delta  \label{eq:cond4}\\ 
    &\mu \PM (Q\Psi(Q))^{-1} \label{eq:cond5}\\ 
    &(1+\eta) \PM Q\sigma^{\alpha} \label{eq:cond6}\\ 
    &\mu \leq h/2. \label{eq:cond7}
  \end{align}
It is also important to notice that the implicit constant appearing in~\eqref{eq:cond6} is independent of the other implicit constants; we may choose it as large as we want without affecting the other implicit constants. In the first three steps, the term $R$ which contains terms of order at least $2$ in $I$ will be ignored, that is we will only consider $\hat{H}=H-R=N+P$.

\medskip

\textit{1. Rational approximations of $\omega_0$ and $\omega \in D_h$}

\medskip

Since $\omega_0$ is non-resonant, given $Q \geq n+2$, we can apply Proposition~\ref{dio}: there exist $n$ rational vectors $v_1, \dots, v_n$, of denominators $q_1, \dots, q_n$, such that $q_1v_1, \dots, q_nv_n$ form a $\Z$-basis of $\Z^n$ and for $j\in\{1,\dots,n\}$,
\[ |\omega_0-v_j|\MP(q_j Q)^{-1}, \quad 1 \leq q_j \MP \Psi(Q).\]
For any $\omega \in D_h$, using~\eqref{eq:cond1} and
$q_j \MP \Psi(Q)$, we have
\begin{equation}\label{shift}
|\omega-v_j| \leq |\omega-\omega_0|+|\omega_0-v_j| \MP h + (q_j Q)^{-1} \MP (Q\Psi(Q))^{-1} + (q_j Q)^{-1} \MP (q_j Q)^{-1}.
\end{equation}

\medskip

\textit{2. Successive rational averagings}

\medskip

Let us set $A_1:=A$, $B_1:=B$ so that
$P_1(\theta,I):=A_1(\theta)+B_1(\theta) \cdot I$ satisfies
$P_1=P$. Recalling that $[\,.\,]_v$ denotes the averaging along the
periodic flow associated to a periodic vector $v \in \R^n$
(see~\eqref{moyv}), we define inductively, for $1 \leq j \leq n$,
\[ A_{j+1}:=[A_j]_{v_j}, \quad B_{j+1}:=[B_j]_{v_j}, \quad
P_{j+1}:=[P_j]_{v_j} \]
so in particular $P_j(\theta,I)=A_j(\theta)+B_j(\theta)\cdot I$ for
$1 \leq j \leq n$. Let us also define $X_j$, for $1 \leq j \leq n$, by
\[ X_j(\theta,I):=C_j(\theta)+D_j(\theta)\cdot I \]
where
\[ C_j(\theta)=q_j \int_0^1 (A_j-A_{j+1})(\theta+tq_jv_j)tdt, \quad
D_j(\theta)=q_j \int_0^1 (B_j-B_{j+1})(\theta+tq_jv_j)tdt. \] 
If we further define $N_j$ by $N_j(I)=e(\omega)+v_j \cdot I$, it is
then easy to check, by a simple integration by parts, that the
equations 
\begin{equation}\label{eqhomo}
\{C_j,N_j\}=A_j-A_{j+1}, \quad \{D_j,N_j\}=B_j-B_{j+1},  \quad 1 \leq
j \leq n, 
\end{equation}
and then
\begin{equation}\label{eqhomo2}
\{X_j,N_j\}=P_j-P_{j+1}, \quad 1 \leq j \leq n,
\end{equation}
are satisfied, where $\{\,.\,,\,.\,\}$ denotes the usual Poisson
bracket. Moreover, we have the estimates
\begin{equation}\label{Estim0}
|A_j|_{\alpha,s} \leq |A|_{\alpha,s} \leq \varepsilon, \quad
|B_j|_{\alpha,s} \leq |B|_{\alpha,s} \leq \mu,  
\end{equation}
and then
\begin{equation}\label{Estim00}
|C_j|_{\alpha,s} \leq q_j|A_j|_{\alpha,s} \leq q_j\varepsilon \MP \Psi(Q)\varepsilon, \quad |D_j|_{\alpha,s} \leq q_j|B_j|_{\alpha,s} \leq q_j \mu \MP \Psi(Q)\mu. 
\end{equation}
Next, for any $0 \leq j \leq n$, define $r_j=r-n^{-1}j\delta$ and
$s_j=s-(2n)^{-1}j\sigma$. We have $r_n=r-\delta \leq r_j \leq r_0=r$
while $s_n=s-\sigma/2 \leq s_j \leq s_0=s$. Let $X_j^t$ be the
time-$t$ map of the Hamiltonian flow of $X_j$. Using~\eqref{Estim00},
together with inequalities~\eqref{eq:cond2}, \eqref{eq:cond3} and
\eqref{eq:cond4}, the condition~\eqref{smallX} and~\eqref{smallXX} of
Proposition~\ref{flotX}, Appendix~\ref{app2}, are satisfied, so the
latter proposition can be applied: for $1 \leq j \leq n$, $X_j^t$ maps
$\T^n \times B_{r_{j}}$ into $\T^n \times B_{r_{j-1}}$ for all
$t \in [0,1]$ and it is of the form
\[ X_j^t(\theta,I)=(\theta+E_j^t(\theta),I+F_j^t(\theta)\cdot
I+G_j^t(\theta)) \] 
with
\begin{equation}\label{estflot1}
\begin{cases}
|E_j^t|_{\alpha,s_j} \leq |D_j|_{\alpha,s_{j-1}} \MP \Psi(Q)\mu \\  
|F_j^t|_{\alpha,s_j} \MP \sigma^{-\alpha}|D_j|_{\alpha,s_{j-1}} \MP
\sigma^{-\alpha}\Psi(Q)\mu \\ 
|G_j^t|_{\alpha,s_j} \MP \sigma^{-\alpha}|C_j|_{\alpha,s_{j-1}} \MP
\sigma^{-\alpha}\Psi(Q)\varepsilon.
\end{cases}
\end{equation} 
Now we define $\Phi^0:=\mathrm{Id}$ to be the identity and inductively
$\Phi^{j}:=\Phi^{j-1} \circ X_{j}^1$ for $1 \leq j \leq n$. Then
$\Phi^{j}$ maps $\T^n \times B_{r_{j}}$ into $\T^n \times B_{r}$ and
one easily checks, by induction using the estimates~\eqref{estflot1},
that $\Phi^{j}$ is still of the form
\[ \Phi^j(\theta,I)=(\theta+E^j(\theta),I+F^j(\theta)\cdot I+G^j(\theta)) \]
with the estimates, for $j=1,...,n$,
\begin{equation}\label{estflot2}
  |E^j|_{\alpha,s_j}  \MP \Psi(Q)\mu, \quad |F^j|_{\alpha,s_j} \MP
  \sigma^{-\alpha}\Psi(Q)\mu, \quad |G^j|_{\alpha,s_j} \MP
  \sigma^{-\alpha}\Psi(Q)\varepsilon. 
\end{equation} 

\medskip

\textit{3. New Hamiltonian}

\medskip

Let us come back to the Hamiltonian $\hat{H}=H-R=N+P=N+P_1$. We claim
that for all $0 \leq j \leq n$, we have
\[ \hat{H} \circ \Phi^j= N+P_{j+1}+P_{j+1}^+, \quad P_{j+1}^+(\theta,I)=A_{j+1}^+(\theta)+B_{j+1}^+(\theta)\cdot I \] 
with the estimates
\begin{equation}\label{estAB}
|A_{j+1}^+|_{\alpha,s_j} \MP (Q\sigma^{\alpha})^{-1}\varepsilon, \quad
|B_{j+1}^+|_{\alpha,s_j} \MP (Q\sigma^{\alpha})^{-1}\mu. 
\end{equation}
Let us prove the claim by induction on $0 \leq j \leq n$. For $j=0$,
we may set $P_1^+:=0$ and there is nothing to prove. So let us assume
that the claim is true for some $j-1 \geq 0$, and we need to show it
is still true for $j \geq 1$. By this inductive assumption, we have
\[ \hat{H} \circ \Phi^j=\hat{H} \circ \Phi^{j-1} \circ X_j^1=(N+P_j+P_j^+) \circ X_j^1 \]
with 
\begin{equation}\label{estABind}
  |A_{j}^+|_{\alpha,s_{j-1}} \MP (Q\sigma^{\alpha})^{-1}\varepsilon,
  \quad |B_{j}^+|_{\alpha,s_{j-1}} \MP (Q\sigma^{\alpha})^{-1}\mu. 
\end{equation}
Let $S_j=\omega\cdot I - v_j\cdot I$ so that $N=N_j+S_j$ and thus
\[ \hat{H} \circ \Phi^j=(N_j+S_j+P_j+P_j^+) \circ X_j^1=(N_j+S_j+P_j)
\circ X_j^1+P_j^+ \circ X_j^1. \]
Let us consider the first summand of the last sum. Using the
equality~\eqref{eqhomo2}, a standard computation based on Taylor's
formula with integral remainder gives
\[ (N_j+S_j+P_j) \circ X_j^1 = N+[P_j]_{v_j} +\tilde{P}_{j+1}= N+P_{j+1}+\tilde{P}_{j+1} \]
with
\[ \tilde{P}_{j+1}=\int_0^1 U_{j+1}^t \circ X_{j}^t dt, \quad U_{j+1}^t:=\{ (1-t)P_{j+1}+t P_j +S_j, X_j \}.  \] 
Clearly, $U_{j+1}^t$ is still of the form
\[ U_{j+1}^{t}(\theta,I)=U_{j+1}^{t}(\theta,0)+\nabla_I
U_{j+1}^t(\theta,0)\cdot I \]
as this is true for $P_j$, $S_j$, $X_j$ and that this form is
preserved under Poisson bracket. Using the estimates for
$P_j(\theta,0)$, $\nabla_I P_j(\theta,0)$, $X_j(\theta,0)$,
$\nabla_I X_j(\theta,0)$ (given respectively in~\eqref{Estim0} and
in~\eqref{Estim00}), the fact that
\[ S_j(\theta,0)=0, \quad \nabla_IS_j(\theta,0)=\omega-v_j \]
with the inequality~\eqref{shift}, and the estimates for the
derivatives and the product of Gevrey functions (given respectively in
Proposition~\ref{derivative}, Corollary~\ref{corderivative} and
Proposition~\ref{produit}, Corollary~\ref{corproduit},
Appendix~\ref{app2}), one finds, for all $t \in [0,1]$
\[ |U_{j+1}^t(\theta,0)|_{\alpha,s_j} \MP
(\sigma^{-\alpha}q_j\varepsilon \mu+\sigma^{-\alpha}q_j\varepsilon
\mu+\sigma^{-\alpha}q_j\varepsilon(q_jQ)^{-1}) \MP
\sigma^{-\alpha}q_j\varepsilon \mu
+(Q\sigma^{\alpha})^{-1}\varepsilon.  \]
Since $q_j \MP \Psi(Q)$, using~\eqref{eq:cond5} the latter estimate reduces to
\[ |U_{j+1}^t(\theta,0)|_{\alpha,s_j} \MP (Q\sigma^{\alpha})^{-1}\varepsilon.  \]
Similarly, one obtains
\[ |\nabla_I U_{j+1}^t(\theta,0)|_{\alpha,s_j} \MP
(Q\sigma^{\alpha})^{-1}\mu.  \]
Then, using the expression of $X_j^t$ and the associated
estimates~\eqref{estflot1}, a direct computation, still
using~\eqref{eq:cond5}, gives
\[ |\tilde{P}_{j+1}(\theta,0)|_{\alpha,s_j} \MP (Q\sigma^{\alpha})^{-1}\varepsilon \]
and
\[ |\nabla_I\tilde{P}_{j+1}(\theta,0)|_{\alpha,s_j} \MP
(Q\sigma^{\alpha})^{-1}\mu.  \]
Using again the estimates of $X_j^t$ given by~\eqref{estflot1},
and the inductive assumption~\eqref{estABind}, we also find
\[ |P_j^+ \circ X_j^1(\theta,0)|_{\alpha,s_j} \MP (Q\sigma^{\alpha})^{-1}\varepsilon \]
and
\[ |\nabla_I(P_j^+ \circ X_j^1)(\theta,0)|_{\alpha,s_j} \MP (Q\sigma^{\alpha})^{-1}\mu.  \]
Eventually, we may define
\[ P_{j+1}^+:=\tilde{P}_{j+1}+P_j^+ \circ X_j^1 \]
so that
\[ \hat{H} \circ \Phi^j= N+P_{j+1}+P_{j+1}^+, \quad P_{j+1}^+(\theta,I)=A_{j+1}^+(\theta)+B_{j+1}^+(\theta)\cdot I \] 
and these last estimates imply that
\begin{equation*}
|A_{j+1}^+|_{\alpha,s_j} \MP (Q\sigma^{\alpha})^{-1}\varepsilon, \quad |B_{j+1}^+|_{\alpha,s_j} \MP (Q\sigma^{\alpha})^{-1}\mu.
\end{equation*}
The claim is proved. So we may set 
\[ \Phi:=\Phi^n, \quad (E,F,G):=(E^n,F^n,G^n), \]
with, as~\eqref{estflot2} tells us with $j=n$ and
$s_n = s - \sigma/2$,
\begin{equation}\label{estflot3}
  |E|_{\alpha,s - \sigma/2}  \MP \Psi(Q)\mu, \quad |F|_{\alpha,s
    \sigma/2} \MP \sigma^{-\alpha}\Psi(Q)\mu, \quad |G|_{\alpha,s -
    \sigma/2} \MP \sigma^{-\alpha}\Psi(Q)\varepsilon. 
\end{equation} 

Observe that $P_{n+1}=[\cdots[P]_{v_1}\cdots]_{v_n}=[P]_{v_1,\dots,v_n}$, and thus by Proposition~\ref{cordio}, $P_{n+1}=[P]$, and as a consequence
\[ \hat{H} \circ \Phi(\theta,I)=e+\omega\cdot I+[A]+[B]\cdot I+A^+_{n+1}(\theta)+B^+_{n+1}(\theta)\cdot I \]
with the estimates
\begin{equation}\label{estAB2}
|A_{n+1}^+|_{\alpha,s-\sigma/2} \MP (Q\sigma^{\alpha})^{-1}\varepsilon, \quad |B_{n+1}^+|_{\alpha,s-\sigma/2} \MP (Q\sigma^{\alpha})^{-1}\mu.
\end{equation}

\medskip

\textit{4. Estimate of the remainder}

\medskip

Now we take into account the remainder term $R$ that we previously ignored: we have $H=\hat{H}+R$, and therefore
\[ H \circ \Phi(\theta,I)=e+\omega\cdot I+[A]+[B]\cdot
I+A^+_{n+1}(\theta)+B^+_{n+1}(\theta)\cdot I+R \circ
\Phi(\theta,I).  \] Let us decompose
\[ R\circ \Phi(\theta,I)=
\underbrace{R\circ\Phi(\theta,0)}_{R_0(\theta)} +
\underbrace{\nabla_I(R \circ \Phi)(\theta,0)}_{R_1(\theta)} \cdot I
+\tilde{R}(\theta,I)\]
and let us define
\[ \tilde{A}:=A^+_{n+1}+R_0, \quad
\tilde{B}:=B^+_{n+1}+R_1. \]
We have $R(\theta,I)=M(\theta,I) \cdot I^2$ and as $H$ and $R$ differ only by terms of order at most one in $I$, $\nabla_I^2 H=\nabla_I^2 R$ so
\[ M(\theta,I)=\int_0^1(1-t)\nabla_I^2 H(\theta,tI)dt=\int_0^1(1-t)\nabla_I^2 R(\theta,tI)dt \]
and therefore $|M|_{\alpha,s}\leq \eta$. Then, as $\Phi(\theta,0)=(\theta+E(\theta),G(\theta))$, we have the expression
\[ R_0(\theta)=R(\Phi(\theta,0))=M(\theta+E(\theta),G(\theta))\cdot G(\theta)^2 \]
and so using the above estimate on
$M$, together with the estimates on $E$, $G$ and the estimates for the product and compostion of Gevrey
functions (given respectively in Proposition~\ref{produit} and
Proposition~\ref{composition}, Appendix~\ref{app2}), we find
\[ |R_0|_{\alpha,s-\sigma/2} \MP \eta|G|^2_{\alpha,s-\sigma/2} \MP
\eta(\sigma^{-\alpha}\Psi(Q)\mu)^2 \varepsilon \MP \eta
(Q\sigma^{\alpha})^{-2}\varepsilon \MP \eta
(Q\sigma^{\alpha})^{-1}\varepsilon. \] 
Then, we have $\nabla_I R(\theta,I)=\hat{M}(\theta,I)\cdot I^2$ with 
\[ \hat{M}(\theta,I)=\int_0^1 \nabla_I^2 H(\theta,tI)dt=\int_0^1\nabla_I^2 R(\theta,tI)dt \]
and hence $|\hat{M}|_{\alpha,s}\leq \eta$. Since 
\begin{equation}\label{F}
|\nabla_I \Phi-\mathrm{Id}|_{\alpha,s-\sigma}=
|F|_{\alpha,s-\sigma} \MP \sigma^{-\alpha}\Psi(Q)\mu\PM 1  
\end{equation}
we obtain, using the fact that
$\varepsilon \leq \mu^2$ and proceeding as before,
\[ |R_1|_{\alpha,s-\sigma/2} \MP \eta|G|_{\alpha,s-\sigma/2} \MP
\eta\sigma^{-\alpha}\Psi(Q)\varepsilon \MP\eta (\sigma^{-\alpha}\Psi(Q)\mu)\mu
\MP\eta (Q\sigma^{\alpha})^{-1}\mu. \] 
These last estimates on $R_0$ and $R_1$, together with~\eqref{estAB2}, imply
\begin{equation*}
|\tilde{A}|_{\alpha,s-\sigma/2} \MP(1+\eta)
(Q\sigma^{\alpha})^{-1}\varepsilon, \quad
|\tilde{B}|_{\alpha,s-\sigma/2} \MP(1+\eta) (Q\sigma^{\alpha})^{-1}\mu. 
\end{equation*}
We can finally now use~\eqref{eq:cond6} to ensure that
\begin{equation}\label{estABF}
|\tilde{A}|_{\alpha,s-\sigma/2} \leq \varepsilon/16, \quad
|\tilde{B}|_{\alpha,s-\sigma/2} \leq \mu/4. 
\end{equation}
It is important to recall here that we may choose the implicit constant in~\eqref{eq:cond6} as large as we want (in order to achieve~\eqref{estABF}) without affecting any of the other implicit constants. Then observe also that $H \circ \Phi$ and $\tilde{R}$ differ only by terms of order at most one in $I$, so
\[   \nabla_I^2 (H\circ \Phi)=\nabla_I^2 \tilde{R}, \quad \nabla_I^2 H=\nabla_I^2 R 
\]
and therefore using the formula for the Hessian of a composition,~\eqref{F} and the fact that $\nabla_I^2 \Phi$ is identically zero, one finds
\begin{equation}\label{estM}
|\nabla_I^2 \tilde{R}-\nabla_I^2 R \circ \Phi|_{\alpha,s-\sigma/2} \MP \eta |F|_{\alpha,\sigma}.
\end{equation}
We also set $\tilde{e}:=e+[A]$ and observe that
\begin{equation}\label{este}
|\tilde{e}-e|_{\alpha,s-\sigma/2} \leq |[A]|_{\alpha,s} \leq |A|_{\alpha,s}.
\end{equation}

\medskip

\textit{5. Change of frequencies and final estimates}

\medskip

Let us now write explicitly the dependence on the parameter $\omega \in D_h$: we have 
\[ H \circ \Phi(\theta,I,\omega)=\tilde{e}(\omega)+(\omega+[B](\omega))\cdot I+\tilde{A}(\theta,\omega)+\tilde{B}(\theta,\omega)\cdot I+ \tilde{R}(\theta,I,\omega).\]
Consider the map $\phi(\omega):=\omega+[B(\omega)]$, it satisfies
\[ |\phi-\mathrm{Id}|_{\alpha,s} \leq |[B]|_{\alpha,s} \leq
|B|_{\alpha,s} \leq \mu \]
and therefore the conditions~\eqref{smallomega} of
Proposition~\ref{propomega} are satisfied: the first condition
of~\eqref{smallomega} follows, from instance, from
condition~\eqref{eq:cond2} and the fact that $\Psi(Q)\geq Q \geq 1$,
whereas the second condition of~\eqref{smallomega} is implied by
condition~\eqref{eq:cond7}. Hence Proposition~\ref{propomega} applies
and one finds a unique
$\varphi \in G_{\alpha,s-\sigma/2}(D_{h/2},D_h)$ such that
$\phi \circ \varphi=\mathrm{Id}$ and
\begin{equation}\label{estom}
|\varphi-\mathrm{Id}|_{\alpha,s-\sigma/2} \leq |\phi-\mathrm{Id}|_{\alpha,s} \leq \mu.
\end{equation}  
We do have $\varphi(\omega)+[B(\varphi(\omega))]=\omega$ and thus,
setting $\mathcal{F}:=(\Phi,\varphi)$, this implies that
\[ H \circ \mathcal{F}(\theta,I,\omega)= H \circ
\Phi(\theta,I,\varphi(\omega)) =
\tilde{e}(\varphi(\omega))+\omega\cdot
I+\tilde{A}(\theta,\varphi(\omega)) +
\tilde{B}(\theta,\varphi(\omega))\cdot I+
\tilde{R}(\theta,I,\varphi(\omega))\] and at the end we set
\[ e^+:=\tilde{e}\circ \varphi, \quad A^+:=\tilde{A} \circ \varphi,
\quad B^+:=\tilde{B} \circ \varphi, \quad R^+:=\tilde{R} \circ
\varphi. \] 
Using once again Proposition~\ref{composition}, the
inequalities~\eqref{estABF},~\eqref{estM} and~\eqref{este} imply
\begin{equation*}
\begin{cases}
|A^+|_{\alpha,s-\sigma}=|\tilde{A} \circ \varphi|_{\alpha,s-\sigma}
\leq |\tilde{A}|_{\alpha,s-\sigma/2} \leq \varepsilon/16, \\  
|B^+|_{\alpha,s-\sigma}=|\tilde{B} \circ \varphi|_{\alpha,s-\sigma}
\leq |\tilde{B}|_{\alpha,s-\sigma/2} \leq \mu/4, \\  
|e^+-e\circ \varphi|_{\alpha,s-\sigma}=|(\tilde{e}-e)\circ
\varphi|_{\alpha,s-\sigma} \leq |\tilde{e}-e|_{\alpha,s-\sigma/2} \leq
|A|_{\alpha,s},  
\\ 
|\nabla_I^2 R^+-\nabla_I^2 R \circ \mathcal{F}|=|(\nabla_I^2 \tilde{R}-\nabla_I^2 R \circ \Phi)\circ \varphi|_{\alpha,s-\sigma}
\leq |\nabla_I^2 \tilde{R}-\nabla_I^2 R \circ \Phi|_{\alpha,s-\sigma/2} \MP \eta |F|_{\alpha,s-\sigma},
\end{cases}
\end{equation*}
which were the estimates~\eqref{stepest2} we needed to prove. The
transformation
$\mathcal{F}=(\Phi,\varphi)=(E,F,G,\varphi) \in \mathcal{G}$ maps
$\T^n \times D_{r-\delta,h/2}$ into $\T^n \times D_{r,h}$ and it
follows from~\eqref{estflot3} and~\eqref{estom} that
\begin{equation*}
\begin{cases}
|E|_{\alpha,s-\sigma} \MP \Psi(Q)\mu, \quad |\nabla E|_{\alpha,s-\sigma} \MP \sigma^{-\alpha}\Psi(Q)\mu, \\
|F|_{\alpha,s-\sigma} \MP \sigma^{-\alpha}\Psi(Q)\mu, \quad |\nabla F|_{\alpha,s-\sigma} \MP \sigma^{-2\alpha}\Psi(Q)\mu, \\
|G|_{\alpha,s-\sigma} \MP \sigma^{-\alpha}\Psi(Q)\varepsilon, \quad |\nabla G|_{\alpha,s-\sigma} \MP \sigma^{-2\alpha}\Psi(Q)\varepsilon,   \\
|\varphi-\mathrm{Id}|_{\alpha,s-\sigma} \leq \mu, \quad |\nabla \varphi-\mathrm{Id}|_{\alpha,s-\sigma} \MP \sigma^{-\alpha} \mu
\end{cases}
\end{equation*}
which were the wanted estimates~\eqref{stepest1}. This concludes the proof. 
\end{proof}

\subsection{Iterations and convergence}\label{sec:iteration}

We now define, for $i \in \N$, the following decreasing geometric sequences:
\begin{equation}\label{defsuite1}
\varepsilon_i:=16^{-i} \varepsilon, \quad \mu_i:=4^{-i}\mu, \quad \delta_i:=2^{-i-2}r, \quad h_i=2^{-i} h.
\end{equation}
Next, for a constant $Q_0 \geq n+2$ to be chosen below, we define $\Delta_i$ and $Q_i$, $i\in \N$, by 
\begin{equation}\label{defsuite2}
\Delta_i=2^i \Delta(Q_0), \quad Q_i=\Delta^{-1}(\Delta_i)=\Delta^{-1}(2^i \Delta(Q_0)), 
\end{equation}
and then we define $\sigma_i$, $i\in \N$, by
\begin{equation}\label{defsuite3}
\sigma_i= C Q_i^{-\frac{1}{\alpha}} 
\end{equation}
where $C \geq 1$ is a sufficiently large constant so that the last condition of~\eqref{cond1} is satisfied for $\sigma=\sigma_0$ and $Q=Q_0$ (and thus for $\sigma=\sigma_i$ and $Q=Q_i$, for any $i \in \N$); clearly, this constant is of the form $C\EP (1+\eta)^{1/\alpha}$. Finally, we define $s_i$ and $r_i$, $i\in\N$, by 
\begin{equation}\label{defsuite4}
s_0=s, \quad s_{i+1}=s_i-\sigma_i, \quad r_0=r, \quad r_{i+1}=r_i-\delta_i.  
\end{equation}
Obviously, we have
\[ \lim_{i \rightarrow +\infty} r_i=r-\sum_{i \in \N}\delta_i=r/2. \]
We claim that, assuming $\Delta^{-1}$ satisfies~\eqref{BR2}, which is equivalent to $\omega_0 \in \mathrm{BR}_\alpha$, we can choose $Q_0$ sufficiently large so that 
\begin{equation*}
\lim_{i \rightarrow +\infty} s_i\geq s/2 \Longleftrightarrow \sum_{i \in \N} \sigma_i \leq s/2.
\end{equation*}
Indeed, since $Q_i=\Delta^{-1}(\Delta_i)=\Delta^{-1}\left(2^{i}\Delta(Q_0)\right)$, we have
\[ \sum_{i\geq 1}Q_i^{-\frac{1}{\alpha}}=\sum_{i\geq 1} \frac{1}{\left(\Delta^{-1}\left(2^{i}\Delta(Q_0)\right)\right)^{\frac{1}{\alpha}}}\leq \int_{0}^{+\infty}\frac{dy}{\left(\Delta^{-1}\left(2^{y}\Delta(Q_0)\right)\right)^{\frac{1}{\alpha}}}=\int_{\Delta(Q_0)}^{+\infty}\frac{(\ln 2)^{-1}dx}{x(\Delta^{-1}(x))^{\frac{1}{\alpha}}}<+\infty \]
where we made a change of variables $x:=2^y \Delta(Q_0)$, and the last integral converges since $\Delta^{-1}$ satisfies~\eqref{BR2}. Now as $\sigma_i =C Q_i^{-\frac{1}{\alpha}}$, we have
\[ \sum_{i \in \N} \sigma_i = C \sum_{i\geq 0}Q_i^{-\frac{1}{\alpha}}=CQ_0^{-\frac{1}{\alpha}}+C \sum_{i\geq 1}Q_i^{-\frac{1}{\alpha}} \leq CQ_0^{-\frac{1}{\alpha}}+C \int_{\Delta(Q_0)}^{+\infty}\frac{(\ln 2)^{-1}dx}{x(\Delta^{-1}(x))^{\frac{1}{\alpha}}}\leq s/2 \]
provided we choose $Q_0$ sufficiently large in order to have 
\begin{equation}\label{eqQ}
Q_0^{-\frac{1}{\alpha}}+\int_{\Delta(Q_0)}^{+\infty}\frac{(\ln 2)^{-1}dx}{x(\Delta^{-1}(x))^{\frac{1}{\alpha}}}\leq (2C)^{-1} s.
\end{equation}
Applying inductively Proposition~\ref{kamstep} we will easily obtain the following proposition.

\begin{proposition}\label{kamiter}
Let $H$ be as in~\eqref{HAM}, with $\omega_0 \in \mathrm{BR}_\alpha$, and fix $Q_0 \geq n+2$ sufficiently large so that~\eqref{eqQ} is satisfied. Assume that
\begin{equation}\label{cond2}
\sqrt{\varepsilon} \leq \mu \leq h/2, \quad \sqrt{\varepsilon} \leq r, \quad h \PM \Delta(Q_0)^{-1}.
\end{equation}
Then, for each $i\in \N$, there exists an $(\alpha,s_i)$-Gevrey smooth transformation 
\[ \mathcal{F}^i=(\Phi^i,\varphi^i)=(E^i,F^i,G^i,\varphi^i) : \T^n \times D_{r_i,h_i} \rightarrow \T^n \times D_{r,h} \in \mathcal{G}, \]
such that $\mathcal{F}^{i+1}=\mathcal{F}^{i} \circ \mathcal{F}_{i+1}$, with
\[ \mathcal{F}_{i+1}=(\Phi_{i+1},\varphi_{i+1})=(E_{i+1},F_{i+1},G_{i+1},\varphi_{i+1}) : \T^n \times D_{r_{i+1},h_{i+1}} \rightarrow \T^n \times D_{r_{i},h_{i}} \in \mathcal{G}, \]
satisfying the following estimates
\begin{equation}\label{stepest1i}
\begin{cases}
|E_{i+1}|_{\alpha,s_{i+1}} \MP \Psi(Q_{i})\mu_{i}, \quad |\nabla E_{i+1}|_{\alpha,s_{i+1}} \MP \sigma_i^{-\alpha}\Psi(Q_i)\mu_i, \\ 
|F_{i+1}|_{\alpha,s_{i+1}} \MP \sigma_i^{-\alpha}\Psi(Q_i)\mu_i, \quad |\nabla F_{i+1}|_{\alpha,s_{i+1}} \MP \sigma_i^{-2\alpha}\Psi(Q_i)\mu_i, \\ 
|G_{i+1}|_{\alpha,s_{i+1}}\MP \sigma_i^{-\alpha}\Psi(Q_i)\varepsilon_i, \quad |\nabla G_{i+1}|_{\alpha,s_{i+1}} \MP \sigma_i^{-2\alpha}\Psi(Q_i)\varepsilon_i, \\ 
|\varphi_{i+1}-\mathrm{Id}|_{\alpha,s_{i+1}} \leq \mu_i, \quad  |\nabla \varphi_{i+1}-\mathrm{Id}|_{\alpha,s_{i+1}} \MP \sigma_i^{-\alpha} \mu_i
\end{cases}
\end{equation}
and such that 
\[ 
  H \circ \mathcal{F}^i(\theta,I,\omega) =
  \underbrace{e^i(\omega)+\omega\cdot I}_{N^i(I,\omega)} +
  \underbrace{A^i(\theta)+B^i(\theta)\cdot I}_{P^i(\theta,I,\omega)} +
  \underbrace{M^i(\theta,I,\omega)\cdot I^2}_{R^i(\theta,I,\omega)} 
\]
with the estimates
\begin{equation}\label{stepest2i}
\begin{cases}
  |A^i|_{\alpha,s_{i+1}} \leq \varepsilon_i, \quad |B^i|_{\alpha,s_{i+1}} \leq \mu_i, \\
  |e^{i+1}-e^{i} \circ \varphi_{i+1}|_{\alpha,s_{i+1}} \leq |A^{i}|_{\alpha,s_{i}}, \\
  |\nabla_I^2 R^{i+1}-\nabla_I^2 R^{i} \circ \mathcal{F}_{i+1}|_{\alpha,s_{i+1}}\MP
  \eta|F_{i+1}|_{\alpha,s_{i+1}}. 
\end{cases}
\end{equation}
\end{proposition}

Let us emphasize that the implicit constants in the above proposition depend only on $n$ and $\alpha$ and are thus independent of $i \in \N$.

\begin{proof}
  For $i=0$, we let $\mathcal{F}^0$ be the identity, $e^0:=e$,
  $A^0:=A$, $B^0:=B$, $R^0:=R$, $M^0:=M$ and there is nothing to
  prove. The general case follows by an easy induction. Indeed, assume
  that the statement holds true for some $i\in \N$ so that
  $H \circ \mathcal{F}^i$ is $(\alpha,s_i)$-Gevrey on the domain
  $\T^n \times D_{r_i,s_i}$. We want to apply
  Proposition~\ref{kamstep} to this Hamiltonian, with
  $\varepsilon=\varepsilon_i$, $\mu=\mu_i$, $r=r_i$, $s=s_i$, $h=h_i$,
  $\sigma=\sigma_{i}$ and $Q=Q_i$. First, by our choice of $Q_0$ and
  $\delta_0$ it is clear that $0<\sigma_{i}< s_i$,
  $0 < \delta_i < r_i$, and $Q_i \geq n+2$. Then we need to check that
  the conditions
  \begin{equation*}
    \sqrt{\varepsilon_i} \leq \mu_i \leq h_i/2,
    \quad \sqrt{\varepsilon_i} \leq r_i, \quad h_i \PM \Delta(Q_i)^{-1}, \quad r_i\mu_i \PM \delta_i
    \Delta(Q_i)^{-1},  \quad 1 \PM Q_i\sigma_i^{\alpha} 
  \end{equation*}
  are satisfied. Since 
  \begin{equation*}
    \Delta(Q_i) = \Delta(\Delta^{-1}(\Delta_i)) =\Delta_i, \quad 2^{-i}r \leq r_i \leq r,
  \end{equation*}
  it is sufficient to check the conditions
  \begin{equation}\label{condit}
    \sqrt{\varepsilon_i} \leq \mu_i \leq h_i/2,
    \quad \sqrt{\varepsilon_i} \leq 2^{-i}r, \quad h_i \PM \Delta_i^{-1}, \quad r\mu_i \PM \delta_i
    \Delta_i^{-1},  \quad 1 \PM Q_i\sigma_i^{\alpha} . 
  \end{equation}
  The last condition of~\eqref{condit} is satisfied, for all
  $i \in \N$, simply by the choice of the constant $C$ in the definition
  of $\sigma_i$. As for the other four conditions of~\eqref{condit},
  using the fact that the sequences $\varepsilon_i$, $\mu_i$, $h_i$,
  $\Delta_i^{-1}$ and $\delta_i$ decrease at a geometric rate with
  respective ratio $1/16$, $1/4$, $1/2$, $1/2$ and $1/2$, it is clear
  that they are satisfied for any $i \in \N$ if and only if they are
  satisfied for $i=0$. The first three conditions of~\eqref{condit} for
  $i=0$ are nothing but~\eqref{cond2}. Moreover, using our choice of
  $\delta_0=r/4$, the fourth condition of~\eqref{condit} for $i=0$ reads
  $\mu \PM \Delta_0^{-1}$ and this also follows from~\eqref{cond2}.

Hence Proposition~\ref{kamstep} can be applied, and all the conclusions of the statement follow at once from the conclusions of Proposition~\ref{kamstep}.
\end{proof}

We can finally conclude the proof of Theorem~\ref{KAMparameter}, by showing that one can pass to the limit $i \rightarrow +\infty$ in Proposition~\ref{kamiter}.

\begin{proof}[Proof of Theorem~\ref{KAMparameter}]
Recall that we are given $\varepsilon>0$, $\mu>0$, $r>0$, $s>0$, $h>0$ and that we define the sequences $\varepsilon_i,\mu_i,\delta_i,h_i$ in~\eqref{defsuite1}, and then we chose $Q_0 \geq n+2$ satisfying~\eqref{eqQ} to define the sequences $\Delta_i,Q_i$ in~\eqref{defsuite2} and $\sigma_i$ in~\eqref{defsuite3} and finally, $s_i$ and $r_i$ were defined in~\eqref{defsuite4}. Moreover, we have 
\begin{equation}\label{limits}
\begin{cases}
\lim_{i \rightarrow +\infty}\varepsilon_i=\lim_{i \rightarrow +\infty}\mu_i=\lim_{i \rightarrow +\infty}h_i=0, \\ 
\lim_{i \rightarrow +\infty}r_i=r-\sum_{i=0}^{+\infty}\delta_i=r/2, \quad \lim_{i \rightarrow +\infty}s_i=s-\sum_{i=0}^{+\infty}\sigma_i\geq s/2 
\end{cases}
\end{equation}
and for later use, let us observe that the following series are convergent and can be made as small as one wishes thanks to condition~\eqref{cond0} of Theorem~\ref{KAMparameter}:
\begin{equation}\label{serie1}
\sum_{i=0}^{+\infty} \sigma_i^{-\alpha}\mu_i \leq \sum_{i=0}^{+\infty} Q_i \mu_i = \sum_{i=0}^{+\infty} (\Psi(Q_i))^{-1}\Delta_i\mu_i \leq 2 (\Psi(Q_0))^{-1}\Delta_0 \mu =2Q_0\mu
\end{equation}
\begin{equation}\label{serie2}
\sum_{i=0}^{+\infty} \mu_i \leq 2\mu 
\end{equation}
\begin{equation}\label{serie3}
\sum_{i=0}^{+\infty} \sigma_i^{-\alpha}\Psi(Q_i)\mu_i \leq \sum_{i=0}^{+\infty} \Delta_i \mu_i \leq 2\Delta_0 \mu=2 Q_0\Psi(Q_0) \mu 
\end{equation}
\begin{equation}\label{serie4}
\sum_{i=0}^{+\infty} \Psi(Q_i)\mu_i \leq \sum_{i=0}^{+\infty} Q_i^{-1}\Delta_i \mu_i \leq 2 Q_0^{-1}\Delta_0 \mu =2\Psi(Q_0)\mu
\end{equation}
\begin{equation}\label{serie5}
\sum_{i=0}^{+\infty} \sigma_i^{-\alpha}\Psi(Q_i)\varepsilon_i \leq \sum_{i=0}^{+\infty} \Delta_i \varepsilon_i \leq 2\Delta_0 \varepsilon=2 Q_0\Psi(Q_0) \varepsilon. 
\end{equation}
Now the condition~\eqref{cond0} of Theorem~\ref{KAMparameter} implies that the condition~\eqref{cond2} of Proposition~\ref{kamiter} is satisfied; what we need to prove is that the sequences given by this Proposition~\ref{kamiter} do convergence in the Banach space of $(\alpha,s/2)$-Gevrey functions. Recall that $\mathcal{F}^0=(E^0,F^0,G^0,\varphi^0)$ is the identity, while for $i \geq 0$, 
\[ (E^{i+1},F^{i+1},G^{i+1},\varphi^{i+1})=\mathcal{F}^{i+1}=\mathcal{F}^{i} \circ \mathcal{F}_{i+1}=(E^{i},F^{i},G^{i},\varphi^{i}) \circ (E_{i+1},F_{i+1},G_{i+1},\varphi_{i+1})\]
from which one easily obtains the following inductive expressions:
\begin{equation}
\begin{cases}
E^{i+1}(\theta,\omega)=E_{i+1}(\theta,\omega)+E^{i}(\theta+E_{i+1}(\theta,\omega),\varphi_{i+1}(\omega)) \\
F^{i+1}(\theta,\omega)=F_{i+1}(\theta,\omega)+F^{i}(\theta+E_{i+1}(\theta,\omega),\varphi_{i+1}(\omega))\cdot (\mathrm{Id}+F_{i+1}(\theta,\omega)) \\
G^{i+1}(\theta,\omega)=(F^{i}(\theta+E_{i+1}(\theta,\omega),\varphi_{i+1}(\omega))+\mathrm{Id})\cdot G_{i+1}(\theta,\omega)+G^{i}(\theta+E_{i+1}(\theta,\omega),\varphi_{i+1}(\omega)) \\
\varphi^{i+1}=\varphi^{i} \circ \varphi_{i+1}.
\end{cases}
\end{equation} 
Let us first prove that the sequence $\varphi^i$ converges. We claim that for all $i \in \N$, we have
\[ |\nabla \varphi^{i}|_{\alpha,s_{i}} \MP \prod_{l=0}^{i}(1+\sigma_l^{-\alpha}\mu_l) \MP 1 \]
where the fact that the last product is bounded uniformly in $i \in \N$ follows from~\eqref{serie1}. For $i=0$, $\varphi^0=\mathrm{Id}$ and there is nothing to prove; for $i \in \N$ since $\varphi^{i+1}=\varphi^\circ +\varphi_{i+1}$ we have
\[ \nabla \varphi^{i+1}=\left(\nabla \varphi^{i} \circ \varphi_{i+1}\right) \cdot \nabla \varphi_{i+1} \]
so that using the estimate for $\varphi_{i+1}$ and $\nabla \varphi_{i+1}$ given in~\eqref{stepest1i}, Proposition~\ref{kamiter}, the claim follows by induction. Using this claim, and writing
\[ \varphi^{i+1}-\varphi^{i}=\varphi^i \circ \varphi_{i+1}-\varphi^i=\left(\int_0^1\nabla \varphi^i \circ (t \varphi_{i+1}+(1-t)\mathrm{Id})dt\right)\cdot (\varphi_{i+1}-\mathrm{Id}) \]
one finds
\[ |\varphi^{i+1}-\varphi^{i}|_{\alpha,s_{i+1}} \MP |\varphi_{i+1}-\mathrm{Id}|_{\alpha,s_{i+1}}, \]
and therefore
\[ |\varphi^{i+1}-\varphi^{i}|_{\alpha,s_{i+1}} \MP \mu_i. \]
Using the convergence of~\eqref{limits} and~\eqref{serie2}, one finds that the sequence $\varphi^i$ converges to a trivial map
\[ \varphi^* : \{\omega_0\} \rightarrow D_h, \quad \varphi^*(\omega_0):=\omega_0^* \]
such that
\[ |\omega_0^*-\omega_0| \MP \mu. \]
Now let us define
\[ V_{i+1}(\theta,\omega):=(\theta+E_{i+1}(\theta,\omega),\varphi_{i+1}(\omega)), \quad V_{i+1}=(\mathrm{Id}+E_{i+1},\varphi_{i+1}) \]
and observe that since $\Psi(Q_i)\geq 1$, then the estimates for $E_{i+1}$, $\nabla E_{i+1}$, $\varphi_{i+1}$ and $\nabla \varphi_{i+1}$ given in Proposition~\ref{kamiter} implies that
\[ |V_{i+1}-\mathrm{Id}|_{\alpha,s_{i+1}} \MP \Psi(Q_{i})\mu_{i}, \quad |\nabla V_{i+1}-\mathrm{Id}|_{\alpha,s_{i+1}} \MP \sigma_i^{-\alpha}\Psi(Q_i)\mu_i.  \]
Using these estimates, and the fact that $E^{i+1}$ can be written as
\[ E^{i+1}=E_{i+1}+E^i \circ V_{i+1} \]
we can proceed as before, using the convergence of~\eqref{serie3} to show first that
\[ |\nabla E^i|_{\alpha,s_i} \MP \sum_{l=0}^i \sigma_l^{-\alpha}\Psi(Q_l)\mu_l \MP 1 \]
and then
\[ |E^{i+1}-E^i|_{\alpha,s_{i+1}} \MP |E_{i+1}|_{\alpha,s_{i+1}} \MP \Psi(Q_i)\mu_i. \] 
Using the convergence of~\eqref{limits} and~\eqref{serie4}, this shows that $E^i$ converges to a map
\[ E^* : \T^n \times \{\omega_0\} \rightarrow \T^n \times D_h \]
such that
\[ |E^*|_{\alpha,s/2} \MP \Psi(Q_0)\mu. \]
For the $F^i$, we do have the expression
\[ F^{i+1}=F_{i+1}+(F^i \circ V_{i+1})\cdot (\mathrm{Id}+F_{i+1}) \]
or alternatively
\[ F^{i+1}=(\mathrm{Id}+F^i \circ V_{i+1})\cdot F_{i+1}+ F^i \circ V_{i+1} \]
and thus
\[ F^{i+1}-F^i=(\mathrm{Id}+F^i \circ V_{i+1})\cdot F_{i+1}+ F^i \circ V_{i+1}-F^i. \]
As before, using the estimates on $F_{i+1}$ and $\nabla F_{i+1}$ given in Proposition~\ref{kamiter}, one shows that
\[ |\nabla F^i|_{\alpha,s_i} \MP \sum_{l=0}^i \sigma_l^{-2\alpha}\Psi(Q_l)\mu_l \]
but however, here, the sum above is not convergent. Yet we do have
\[ \sigma_i^{\alpha}|\nabla F^i|_{\alpha,s_i} \MP \sigma_i^{\alpha}\sum_{l=0}^i \sigma_l^{-2\alpha}\Psi(Q_l)\mu_l \MP \sum_{l=0}^i \sigma_l^{-\alpha}\Psi(Q_l)\mu_l \MP 1  \]
from~\eqref{serie4} and using the fact that the estimate for $V_{i+1}$ can be written as
\[ |V_{i+1}-\mathrm{Id}|_{\alpha,s_{i+1}} \MP \sigma_i^{\alpha}\sigma_i^{-\alpha}\Psi(Q_{i})\mu_{i} \]
one obtains
\[ |F^i \circ V_{i+1}-F^i|_{\alpha,s_{i+1}}  \MP \sigma_i^{-\alpha}\Psi(Q_{i})\mu_{i}.  \]
By induction, one shows that 
\[ |F^i|_{\alpha,s_i} \MP \sum_{l=0}^i \sigma_l^{-\alpha}\Psi(Q_l)\mu_l \MP 1 \]
from which one obtains
\[ |\mathrm{Id}+F^i \circ V_{i+1}|_{\alpha,s_i} \MP 1 \]
and as a consequence, 
\[ |F^{i+1}-F^i|_{\alpha,s_{i+1}}  \MP \sigma_i^{-\alpha}\Psi(Q_{i})\mu_{i}. \]
Using the convergence of~\eqref{limits} and~\eqref{serie3}, this shows that $F^i$ converges to a map
\[ F^* : \T^n \times \{\omega_0\} \rightarrow \T^n \times D_h \]
such that
\[ |F^*|_{\alpha,s/2} \MP Q_0\Psi(Q_0)\mu. \]
For $G^i$, we have the expression
\[ G^{i+1}=(F^i \circ V_{i+1}+\mathrm{Id})\cdot G_{i+1}+G^i \circ V_{i+1} \]
and thus
\[ G^{i+1}-G^i=(F^i \circ V_{i+1}+\mathrm{Id})\cdot G_{i+1}+G^i \circ V_{i+1}-G^i. \]
Proceeding exactly as we did for $E^i$ and $F^i$, using the convergence of~\eqref{limits},~\eqref{serie3} and~\eqref{serie5}, one finds that $G^i$ converges to a map
\[ G^* : \T^n \times \{\omega_0\} \rightarrow \T^n \times D_h \]
such that
\[ |G^*|_{\alpha,s/2} \MP Q_0\Psi(Q_0)\varepsilon. \]
In summary, the map $\mathcal{F}^i$ converges to a map
\begin{equation*}
\mathcal{F}^* : \T^n \times D_{r/2} \times \{\omega_0\} \rightarrow \T^n \times D_{r,h} 
\end{equation*}
which belongs to $\mathcal{G}$, of the form
\[ 
\begin{cases}
\mathcal{F}^*(\theta,I,\omega_0)=(\Phi^*_{\omega_0}(\theta,I),\omega_0^*), \\
\Phi_{\omega_0}^*(\theta,I)=(\theta+E^*(\theta),I+F^*(\theta)\cdot I +G^*(\theta))
\end{cases}
 \]
with the estimates
\begin{equation}\label{estF}
|E^*|_{\alpha,s/2} \MP \Psi(Q_0)\mu, \quad 
|F^*|_{\alpha,s/2} \MP Q_0\Psi(Q_0)\mu, \quad |G^*|_{\alpha,s/2} \MP Q_0\Psi(Q_0)\varepsilon, \quad 
|\omega_0^*-\omega_0| \MP \mu.
\end{equation} 
Then from the estimates
\[ |A^i|_{\alpha,s_i} \leq \varepsilon_i, \quad |B^i|_{\alpha,s_i} \leq \mu_i, \]
given in~\eqref{stepest2i}, Proposition~\ref{kamiter}, and the convergence~\eqref{limits}, it follows that both $A^i$ and $B^i$ convergence to zero. Next from the estimates
\begin{equation*}
\begin{cases}
  |e^{i+1}-e^{i} \circ \varphi_{i+1}|_{\alpha,s_{i+1}} \leq |A^{i}|_{\alpha,s_{i}}, \\
  |\nabla_I^2R^{i+1}-\nabla_I^2R^{i} \circ  \mathcal{F}_{i+1}|_{\alpha,s_{i+1}}\MP \eta
  |F_{i+1}|_{\alpha,s_{i+1}} 
\end{cases}
\end{equation*}
still given in~\eqref{stepest2i}, Proposition~\ref{kamiter}, one can
prove in the same way as we did before, that $e^i$
converges to a trivial map
\[ e^* : \{\omega_0\} \rightarrow D_h, \quad e^*(\omega_0):=e_0^* \]
such that
\begin{equation}\label{estFF}
|e_0^*-e_{\omega_0^*}| \MP \varepsilon 
\end{equation}
whereas $M^i$ converges to a map
\[ M^* : \T^n \times D_{r/2} \times \{\omega_0\} \rightarrow \T^n
\times D_{r,h} \] 
such that, setting $R^*(\theta,I)=M^*(\theta,I)I\cdot I$,
\begin{equation}\label{estFFFF}
|\nabla_I^2 R^*-\nabla_I^2 R_{\omega_0^*}|_{\alpha,s/2} \MP \eta Q_0\Psi(Q_0)\mu.
\end{equation} 
Therefore we have
\[ H \circ \mathcal{F}^*(\theta,I,\omega_0)= H_{\omega_0^*} \circ \Phi_{\omega_0}^*(\theta,I)=e_0^*+\omega_0 \cdot I
+R^*(\theta,I), \]
which, together with the previous estimates~\eqref{estF},~\eqref{estFF} and~\eqref{estFFFF}, is what we wanted to prove.
\end{proof}

\section{Proof of Theorem~\ref{destruction}, following Bessi}\label{sec:bessi}

The goal of this short section is to show how Theorem~\ref{destruction} follows directly from the work of Bessi in~\cite{Bes00}.

In Bessi, one starts with a non-resonant vector $\omega \in \R^n$ which is assumed to be ``exponentially Liouville" in the following sense: there exists $s_0>0$ and a sequence $k_j \in \Z^n$ with $|k_j| \rightarrow +\infty$ as $j \rightarrow +\infty$ for which
\begin{equation}\label{condBessi}
0<|k_j\cdot \omega| \leq e^{-s_0|k_j|}\tag{$C_{1,s_0}$}.
\end{equation}
Given this sequence of $k_j \in \Z^n$, one can find another sequence $\tilde{k}_j \in \Z^n$ such that for all $j \in \N$, $|\tilde{k}_j|\leq |k_j|$, $\tilde{k}_j \cdot k_j=0$ and $|\tilde{k}_j\cdot \omega| \geq c|\tilde{k}_j|$ for some constant $c>0$ independent of $j$. 

Then one defines the following sequence of Hamiltonians on $\R^{n}/(2\pi\Z^n) \times \R^n$ (which are similar to the Hamiltonian considered by Arnold in~\cite{Arn64}):
\begin{equation}\label{HamBessi}\tag{$H_{1,j,s}$}
\begin{cases}
H^{1,j}_{\varepsilon,\mu}(\theta,I):=\frac{1}{2}I\cdot I+F^{1,j}_{\varepsilon,\mu}(\theta), \; (\theta,I) \in \R^{n}/(2\pi\Z^n) \times \R^n \\
F^{1,j}_{\varepsilon,\mu}(\theta):=\varepsilon\nu_{1,j,s}(1-\cos(k_j\cdot\theta))(1+\mu\tilde{\nu}_{1,j,s}\cos(\tilde{k}_j\cdot\theta)) \\
0<\varepsilon \leq 1, \; 0 < \mu \leq 1, \; \nu_{1,j,s}:=e^{-s|k_j|}, \; \tilde{\nu}_{1,j,s}:=e^{-s|\tilde{k}_j|}.  
\end{cases}
\end{equation}
Observe that the only role of the sequences $\nu_{1,j,s}$ and $\tilde{\nu}_{1,j,s}$ is to ensure that the sequence of perturbations $F^{1,j}_{\varepsilon,\mu}$ satisfy, for all $j \in \N $ and all $0 \leq \mu \leq 1$:
\[ |F^{1,j}_{\varepsilon,\mu}|_s:=\sup_{\theta \in \C^n /(2\pi\Z^n), \; |\mathrm{Im}(\theta)\leq s|}|F^{1,j}_{\varepsilon,\mu}(\theta)|\leq 4\varepsilon. \]
In~\cite{Bes00}, Bessi proved the following theorem.

\begin{theorem}[Bessi]
Assume that $\omega \in \R^n$ satisfy~\eqref{condBessi}. Then, if $s_0>s$, for any $0 \leq \varepsilon \leq 1$, there exists $\mu_\varepsilon>0$ and $j_\varepsilon \in \N$ such that for any $0<\mu \leq \mu_\varepsilon$ and any $j \geq j_\varepsilon$, the Hamiltonian system defined in~\eqref{HamBessi} does not have any invariant torus $\mathcal{T}$ satisfying
\begin{itemize}
\item[$(i)$] $\mathcal{T}$ projects diffeomorphically to $\T^n$;
\item[$(i)$] There is a $C^1$ diffeomorphism between $\T^n$ and $\mathcal{T}$ which conjugates the flow on $\mathcal{T}$ to the linear flow on $\T^n$ of frequency $\omega$. 
\end{itemize}
\end{theorem}

It is clear that it is the regularity of the perturbation, here the analyticity which causes the exponential decay of the Fourier coefficients, that forces the condition~\eqref{condBessi}. If the perturbation is assumed to be only of class $C^r$ for some $r \in \N$, then~\eqref{condBessi} can be weakened to cover frequencies $\omega$ which are Diophantine with an exponent $\tau$ which is related to $r$ (this can be obtained from Bessi's work; one can find a better quantitative result in~\cite{CW13}, which also uses ideas of~\cite{Bes00}).

Here we would like to consider the case where the perturbation is $\alpha$-Gevrey; we will consider a slight modification of the family of Hamiltonians~\eqref{HamBessi} to a family of Hamiltonians~\eqref{HamBessiG} depending on $\alpha \geq 1$, which are still analytic but for which the perturbation are bounded and small only in a $\alpha$-Gevrey norm.

First we need to compute the Gevrey of the function $P_k(\theta):=\cos(k\cdot \theta)$ for an arbitrary $k \in \Z^n$. Using the fact that $(l+1)^2 \leq 4^l$ for any $l \in \N$ we have
\begin{equation*}
|P_k|_{\alpha,s}=c\sup_{l \in \N}\frac{(l+1)^2s^{\alpha l} |k|^l}{l!^\alpha} \leq c \sup_{l \in \N}\frac{(4|k|s^\alpha)^l}{l!^\alpha} \leq c \left(\sup_{l \in \N}\frac{((4|k|)^{1/\alpha}s)^l}{l!}\right)^{\alpha} \leq c e^{s\alpha(4|k|)^{\frac{1}{\alpha}}}.
\end{equation*}
Now, given $\alpha \geq 1$, we introduce the following condition on the non-resonant vector $\omega \in \R^n$: there exists $s_0>0$ and a sequence $k_j \in \Z^n$ with $|k_j| \rightarrow +\infty$ as $j \rightarrow +\infty$ for which
\begin{equation}\label{condBessiG}
0<|k_j\cdot \omega| \leq e^{-s_0\alpha(4|k_j|)^\frac{1}{\alpha}}\tag{$C_{\alpha,s_0}$}.
\end{equation}
For $\alpha=1$, this condition reduces to~\eqref{condBessi}. Then we consider the following modified sequence of Hamiltonians, which once again corresponds exactly to~\ref{HamBessi} for $\alpha=1$:
\begin{equation}\label{HamBessiG}\tag{$H_{\alpha,j,s}$}
\begin{cases}
H^{\alpha,j}_{\varepsilon,\mu}(\theta,I):=\frac{1}{2}I\cdot I+F^{\alpha,j}_{\varepsilon,\mu}(\theta), \; (\theta,I) \in \R^{n}/(2\pi\Z^n) \times \R^n \\
F^{\alpha,j}_{\varepsilon,\mu}(\theta):=\varepsilon\nu_{\alpha,j,s}(1-\cos(k_j\cdot\theta))(1+\mu\tilde{\nu}_{\alpha,j,s}\cos(\tilde{k}_j\cdot\theta)) \\
0<\varepsilon \leq 1, \; 0 < \mu \leq 1, \; \nu_{\alpha,j,s}:=e^{-s\alpha(4|k_j|)^{\frac{1}{\alpha}}}, \; \tilde{\nu}_{\alpha,j,s}:=e^{-s\alpha(4|\tilde{k}_j|)^{\frac{1}{\alpha}}}.  
\end{cases}
\end{equation}
With these choices of $\nu_{\alpha,j,s}$ and $\tilde{\nu}_{\alpha,j,s}$ we have that, for all $j \in \N $ and all $0 \leq \mu \leq 1$:
\[ |F^{\alpha,j}_{\varepsilon,\mu}|_{s,\alpha} \leq C\varepsilon \]
for some constant $C>1$ independent of $\varepsilon$ and $\mu$. The argument of Bessi goes exactly the same of way for this family of Hamiltonians~\eqref{HamBessiG} under the condition~\eqref{condBessiG}, and thus we have the following statement.

\begin{theorem}\label{BessiG}
Assume that $\omega \in \R^n$ satisfy~\eqref{condBessiG}. Then, if $s_0> s$, for any $0 \leq \varepsilon \leq 1$, there exists $\mu_\varepsilon>0$ and $j_\varepsilon \in \N$ such that for any $0<\mu \leq \mu_\varepsilon$ and any $j \geq j_\varepsilon$, the Hamiltonian system defined in~\eqref{HamBessiG} does not have any invariant torus $\mathcal{T}$ satisfying
\begin{itemize}
\item[$(i)$] $\mathcal{T}$ projects diffeomorphically to $\T^n$;
\item[$(i)$] There is a $C^1$ diffeomorphism between $\T^n$ and $\mathcal{T}$ which conjugates the flow on $\mathcal{T}$ to the linear flow on $\T^n$ of frequency $\omega$. 
\end{itemize}
\end{theorem}

Now Theorem~\ref{BessiG} implies Theorem~\ref{destruction}, as if $\omega$ satisfies~\eqref{conddestruction}, then it satisfies~\eqref{condBessiG} for some $s_0>0$ and it is sufficient to consider a Hamiltonian system as in~\eqref{HamBessiG} with $s<s_0$.     

\appendix

\section{On the $\alpha$-Bruno-R{\"u}ssmann condition}
\label{app1}

Recall from Section~\ref{sec:aBR} that the arithmetic function on
$[1,+\infty)$ associated with a vector $\omega \in \R^n$ is given by
$$\Psi_\omega(Q) = \max \left\{|k \cdot \omega|^{-1}, \; k \in
  \Z^n, \; 0 < |k| \leq Q \right\}$$
and that we have introduced a continuous version $\Psi$ which satisfies~\eqref{foncpsi}. We have further defined $\Delta$ by \(\Delta(Q) = Q\Psi(Q)\) and its functional inverse $\Delta^{-1}$.

Also recall that, by definition, if $\alpha\geq 1$,
$\mathrm{A}_\alpha$ consists of those vectors $\omega$ for which
\begin{equation*}
  \int_{\Delta
    (1)}^{+\infty}\frac{dx}{x(\Delta^{-1}(x))^{\frac{1}{\alpha}}} <
  \infty, \tag{$\mathrm{A}_\alpha$}
\end{equation*}  
whereas $\mathrm{BR}_\alpha$ consists of vectors $\omega$ satisfying
the $\alpha$-Bruno-R{\"u}ssmann condition:
\begin{equation*}
  \int_{1}^{+\infty}\frac{\ln(\Psi_\omega(Q))}{Q^{1+\frac{1}{\alpha}}}dQ <
  \infty \Leftrightarrow \int_{1}^{+\infty}\frac{\ln(\Psi(Q))}{Q^{1+\frac{1}{\alpha}}}dQ <
  \infty. \tag{$\mathrm{BR}_\alpha$} 
\end{equation*} 
In the proof of Theorem~\ref{KAMparameter}, we use the following fact. 

\begin{proposition}
  \label{lemmeapp}%
 For any $\alpha \geq 1$, $\mathrm{A}_\alpha = \mathrm{BR}_\alpha$.
\end{proposition}

\begin{proof}
  We aim at showing that the two integrals
  \[\int^{+\infty}\frac{dx}{x(\Delta^{-1}(x))^{\frac{1}{\alpha}}}
  \quad \mbox{and} \quad
  \int^{+\infty}\frac{\ln(\Psi(Q))}{Q^{1+\frac{1}{\alpha}}} \, dQ \]
  converge or diverge simultaneously. Equivalently, choosing for
  example $t=\Delta^{-1}(1)$ and, in the first integral, making the
  change of variable $x=\Delta(Q)$, we may compare the following two
  quantities:
  \[a_\alpha = \int_t^\infty \frac{d\Delta(Q)}{Q^{\frac{1}{\alpha}}
    \Delta(Q)} \quad \mbox{and} \quad b_\alpha = \int_t^\infty
  \frac{\ln \Delta(Q)}{Q^{1+\frac{1}{\alpha}}}\, dQ.\]

  A (Riemann-Stieltjes) integration by part shows that, if $T>1$,
  \begin{equation}
    \label{eq:ipp}
    \int_t^T \frac{d\Delta(Q)}{Q^{\frac{1}{\alpha}} \Delta(Q)} =
    \frac{\ln \Delta(T)}{T^{\frac{1}{\alpha}}} + \frac{1}{\alpha}
    \int_t^T \frac{\ln \Delta(Q)}{Q^{1+\frac{1}{\alpha}}}\, dQ.
  \end{equation}
  On the one hand, the two integrals in this equality have a (possibly infinite) limit as
  $T$ tends to $+\infty$, and $\frac{\ln
    \Delta(T)}{T^{\frac{1}{\alpha}}} \geq 0$, thus \eqref{eq:ipp}
  yields, as $T$ tends to infinity,
  \begin{equation*}
    a_\alpha \geq \frac{b_\alpha}{\alpha}.
  \end{equation*}
  On the other hand, since $\Delta$ is increasing,
  \[\frac{\ln \Delta(T)}{T^{\frac{1}{\alpha}}}  = \frac{\ln
    \Delta(T)}{\alpha} \int_T^{+\infty}
  \frac{dQ}{Q^{1+\frac{1}{\alpha}}} \leq \frac{1}{\alpha}
  \int_T^{+\infty} \frac{\ln \Delta(Q)}{Q^{1+\frac{1}{\alpha}}}
  \, dQ \]
  so, letting $T$ tend to $+\infty$,  \eqref{eq:ipp} entails
  \begin{equation*}
    a_\alpha \leq \frac{b_\alpha}{\alpha}.
  \end{equation*}
  So, $a_\alpha =\frac{b_\alpha}{\alpha}$, whence the conclusion.
\end{proof}

\begin{remark}\label{rem}
From the proof above, one easily see that if $\omega \in \mathrm{BR}_\alpha$, then
\[ \lim_{Q \rightarrow +\infty} \frac{\ln(\Delta(Q))}{Q^{1/\alpha}}=0.\]
But $\ln(\Delta(Q))=\ln(Q\Psi(Q))=\ln Q + \ln(\Psi(Q))$ and therefore
\[ \lim_{Q \rightarrow +\infty} \frac{\ln(\Psi(Q))}{Q^{1/\alpha}}=0\]
which means that $\omega$ satisfies~\ref{coho}.
\end{remark}

We refer to~\cite[Proposition~2.2]{BF13} for related, more precise
results. In the next lemma, we give alternative characterizations of
the $\alpha$-BR condition, so as to facilitate comparisons with
other arithmetic conditions.


\begin{lemma}
  \label{lm:comparison}%
  Let $\alpha \geq 1$, $\beta = 1 + \frac{1}{\alpha}$ and
  $\omega \in \R^n$ non-resonant. The following conditions are
  equivalent to each other:
  \begin{enumerate}
  \item \(\int_1^{+\infty} \frac{\ln
      \Psi_\omega(Q)}{Q^\beta} \, dQ < \infty\) 
  \item \(\sum_{Q\geq1} \frac{\ln\Psi_\omega(Q)}{Q^\beta} <
    \infty\) 
  \item \(\sum_{l\geq 0} \frac{\ln\Psi_\omega(2^l)}{2^{l/\alpha}} < \infty\)
  \item if $n=2$ and $\omega = (1,\nu)$, 
    \(\sum_{k\geq 1} \frac{\ln q_{k+1}}{q_k^{1/\alpha}} < \infty,\)
    where the $q_k$'s are the main denominators of the continued
    fraction of $\nu$.
  \end{enumerate}
\end{lemma}

In the case $\alpha=1$, the equivalence $(1 \Leftrightarrow 2)$ is
proved in \cite{Rus01}, whereas
$(2 \Leftrightarrow 3 \Leftrightarrow 4)$ is proved in in~\cite{GM10}.

\begin{proof}
  $(1 \Leftrightarrow 2)$ As already noticed in Section~\ref{sec:aBR},
  $\Psi_\omega$ is constant on intervals of the form $[Q,Q+1)$,
  $Q \in \N_*$.  So,
  \[\int_1^\infty \frac{\ln \Psi_\omega(Q)}{Q^\beta} \, dQ =
  \sum_{Q=1}^{+\infty} \ln \Psi_\omega(Q) \int_{Q}^{Q+1} q^{-\beta} \, dq =
  \alpha \sum_{Q \geq 1} \ln \Psi_\omega(Q) \left(Q^{-1/\alpha} -
    (Q+1)^{-1/\alpha}\right).\]
  Since the general term of this series is positive and
  $Q^{-1/\alpha} - (Q+1)^{-1/\alpha} \sim_{Q \to +\infty} Q^{-\beta}$,
  the first two conditions are equivalent to one another.

  $(2 \Leftrightarrow 3)$ We want to show that
  $f_2 = \sum_{Q \geq 1} \frac{\ln \Psi_\omega(Q)}{Q^\beta}$ and
  $f_3 = \sum_{l\geq 0} \frac{\ln\Psi_\omega(2^l)}{2^{l/\alpha}}$ are
  simultaneously finite or infinite. Since $\Psi_\omega$ is non-decreasing,
  \[f_2 \geq \sum_{l\geq 0} \ln \Psi_\omega(2^l) \sum_{2^l \leq Q \leq
    2^{l+1}-1} \frac{1}{Q^\beta},\]
  where
  \[\sum_{2^l \leq Q \leq 2^{l+1}-1} \frac{1}{Q^\beta} \geq
  \int_{2^l}^{2^{l+1}} \frac{dQ}{Q^\beta} \geq
  \frac{\alpha}{2} \, \frac{1}{2^{l/\alpha}},\]
  so that
  \[f_2 \geq \frac{\alpha}{2} f_3.\]
  
  Conversely, again using the fact that $\Psi_\omega$ is non-decreasing, 
  \[f_2 \leq \sum_{l\geq 0} \ln \Psi_\omega(2^{l+1}) \sum_{2^l \leq Q \leq
    2^{l+1}-1} \frac{1}{Q^\beta},\]
  where
  \[\sum_{2^l \leq Q \leq 2^{l+1}-1} \frac{1}{Q^\beta} \leq \frac{1}{2^{l\beta}}+
  \int_{2^l}^{2^{l+1}-1} \frac{dQ}{Q^\beta} \leq
  \, \frac{1}{2^{l/\alpha}}\left(\frac{1}{2^l}+\alpha\right) \leq
  \frac{1+\alpha}{2^{l/\alpha}}=\frac{2^{1/\alpha}(1+\alpha)}{2^{(l+1)/\alpha}}\]
  so that \[f_2 \leq 2^{1/\alpha}(1+\alpha) f_3.\]

  $(2 \Leftrightarrow 4)$ Now assume $n=1$ and $\omega = (1,\nu)$,
  $\nu \in \R$. Let $(p_k/q_k)$ be the sequence of convergents of
  $\nu$: one has then the following explicit expression
  $$\Psi_\omega(Q)=|q_k\nu-p_k|^{-1}, \quad q_k \leq Q < q_{k+1}$$
as was proved in~\cite{BF13}, Proposition 2.7.   
  Hence,
  \[
  S :=\sum_{Q \geq 1} \frac{\ln(\Psi_\omega(Q))}{Q^\beta}= \sum_{k\geq 1} \left(\ln |q_k\nu-p_k|^{-1} \sum_{Q =
      q_k}^{q_{k+1}-1} \frac{1}{Q^\beta}\right). 
  \]
  Bounding the first factor by
  \[ \ln q_{k+1} \leq \ln |q_k\nu-p_k|^{-1} \leq \ln (2q_{k+1}) \]
  and the second one by
  \[\frac{\alpha}{q_k^{1/\alpha}} - \frac{\alpha}{q_{k+1}^{1/\alpha}} \leq
  \sum_{Q = q_k}^{q_{k+1}-1} \frac{1}{Q^\beta} \leq
  \frac{1+\alpha}{q_k^{1/\alpha}}\]
  (obtained by comparing the series with an integral)
  yields
  \begin{equation*}
    \alpha \sum_{k\geq 1} \left( \frac{\ln q_{k+1}}{q_k^{1/\alpha}} -
      \frac{\ln q_{k+1}}{q_{k+1}^{1/\alpha}} \right)  \leq S \leq
    (1+\alpha) \sum_{k\geq 1} \frac{\ln (2q_{k+1})}{q_k^{1/\alpha}}=(1+\alpha) \sum_{k\geq 1}\left( \frac{\ln 2}{q_k^{1/\alpha}}+\frac{\ln (q_{k+1})}{q_k^{1/\alpha}}\right).
  \end{equation*}
Since the sequence $(q_k)_{k \geq 0}$ increases at least geometrically, namely $q_k \geq 2^{k/2}$ for any $k \geq 0$, we have
\[ \sum \frac{\ln q_{k+1}}{q_{k+1}^{1/\alpha}} < \infty, \quad \sum \frac{1}{q_k^{1/\alpha}} < \infty   \] 
and therefore
\[ S< \infty \Leftrightarrow \sum \frac{\ln q_{k+1}}{q_k^{1/\alpha}} < \infty.\]
\end{proof}

\begin{example}
  \label{ex:BRa}%
  Let $\nu =\sum_{n\geq 1} 10^{-n!}$. Recall from \cite[Theorem
  30]{K63} that if $p/q \in \Q$ satisfies
  \[| p - q\nu | < \frac{1}{2q},\]
  then $p/q$ is a best approximation of $\nu$.  So, we see that
  $\sum_{1 \leq n \leq Q} 10^{-n!}$ are best approximations, and thus
  $q_n = 2^{n!}$. The 4th criterion of Lemma~\ref{lm:comparison} thus
  tells us that $\omega = (1,\nu) \in \mathrm{BR}_\alpha$ for all
  $\alpha \geq 1$, while $\omega$ is not Diophantine.
\end{example}

\section{Gevrey estimates}\label{app2}

Let us start by recalling some notations and definitions. Given an
integer $m \geq 1$ and $k=(k_1,\dots,k_m) \in \N^m$, we define
\[ |k|=\sum_{i=1}^m k_i, \quad k!=\prod_{i=1}^m k_i!. \]
Given $x \in \R^m$, we set
\[ x^k=\prod_{i=1}^m x_i^{k_i}. \]

Let $K$ be a compact set of the form
\[K = \T^{m_1} \times \bar B^{m_2}, \quad m_1 + m_2 = m,\]
where $\bar B^{m_2}$ is the closure of an open subset $B^{m_2}$ of
$\R^{m_2}$. Let $f : K \rightarrow \R$ be a smooth function, meaning
that $f$ extends smoothly to an open neighborhood of $K$. Such an
extension is by no means unique, but note that, by continuity, the
partial derivatives of $f$ over $K$, at any order, do not depend on
the extension. For $a \in K$ and $k \in \N^m$ we set
\begin{equation}
  \partial^k f(a)= \partial_{x_1}^{k_1} \cdots \partial_{x_m}^{k_m} f(a).
\end{equation}

Given real numbers $\alpha \geq 1$ and $s>0$, the function $f$ is said
to be $(\alpha,s)$-Gevrey if
\begin{equation}\label{defnorm}
|f|_{\alpha,s} := c\sup_{a \in K}\left(\sup_{k \in \N^m}
  \frac{(|k|+1)^2s^{\alpha|k|}|\partial^k
    f(a)|}{|k|!^\alpha}\right)<\infty, \quad c:=4\pi^2/3. 
\end{equation}
The space of such functions will be denoted by $G_{\alpha,s}(K)$, and
equipped with the above norm, it is a Banach space. Our definition of
Gevrey norm is not quite standard, but up to decreasing or increasing
the parameter $s$, it is comparable to the Gevrey norms that have been
used in Hamiltonian perturbation theory (as for instance
in~\cite{MS02} or in~\cite{Pop04}). On the one hand, the role of the
factor $(|k|+1)^2$ is to simplify the estimates for the product and
composition of Gevrey functions (see respectively
Lemmas~\ref{lemmeprod} and~\ref{lemmecomp}). On the other hand, the
role of the normalizing constant $c>0$ in the definition is to ensure
that $G_{\alpha,s}(K)$ is a Banach algebra (see
Lemma~\ref{lemmeprod}).

The above definition can be extended to vector-valued functions
$f=(f_i)_{1 \leq i \leq p} : K \rightarrow \R^p$ for $p \geq 1$ by
setting
\[ |\partial^k f(a)|:=\sum_{1 \leq i \leq p}|\partial^k f_i(a)|, \quad
a \in K \]
in~\eqref{defnorm}. The space of such vector-valued functions is still
a Banach space with the above norm, and it will be denoted by
$G_{\alpha,s}(K,\R^p)$. The case of matrix-valued functions, say with
values in the space $M_{m,p}(\R)$ of matrix with $m$ rows and $p$
columns, is reduced to the case of vector-valued functions by simply
identifying $M_{m,p}(\R)$ to $\R^{mp}$.

\subsection{Majorant series and Gevrey functions}

The definition of Gevrey functions can be conveniently reformulated in terms of majorant series with one variable (see \cite{Kom79}, \cite{Kom80} and also \cite{SCK03}). 

But first let us consider a formal power series in $m$ variables $X=(X_1,\dots,X_m)$ with coefficients in a normed real vector space $(E,|\,.\,|_E)$, which is a formal sum of the form
\[ A(X)=\sum_{k \in \N^m}A_k X^k, \quad A_k=A_{k_1,\dots,k_m} \in E.  \]
Such a formal series is said to be majorized by another formal power series with real non-negative coefficients
\[ B(X)=\sum_{k \in \N^m}B_k X^k, \quad B_k=B_{k_1,\dots,k_m} \in \R,  \]
and we write $A \ll B$, if
\begin{equation}\label{major1}
|A_k|_E \leq B_k, \quad \forall \, k \in \Z^m.
\end{equation}
Next, following~\cite{SCK03}, we introduce a notion of a smooth
function being majorized by a formal power series in one variable. So
let $f : K \rightarrow \R^p$ be a smooth function, and $F$ be a formal
power series in one variable with non-negative coefficients, that we
shall write as
\[ F(X)=\sum_{l=0}^{+\infty}\frac{F_l}{l!}X^l. \]
We will say that $f$ is majorized by $F$ on $K$, and we will write
$f \ll_K F$ (or $f(x) \ll_K F(X)$), if for all $a \in K$ and all
$k \in \N^m$, we have
\begin{equation}\label{major2}
|\partial^k f(a)| \leq F_{|k|}. 
\end{equation}  
To better understand this definition, recall that given
$f : K \rightarrow \R^p$ and $a \in K$, we can define its formal
Taylor series at $a$ by
\[ T_af(X):=\sum_{k \in \N^m}\frac{\partial^kf(a)}{k!} X^k, \]
which is a formal power series in $m$ variables that takes values in
$\R^p$. To a formal series $F$ in one variable, one can associate a
formal series $\hat{F}$ in $m$ variables simply by setting
\[ \hat{F}(X_1,\dots,X_m):=F(X_1+\cdots+X_m). \]
If is then easy to check that $f \ll_K F$, in the sense
of~\eqref{major2}, if and only if for all $a \in K$,
$T_af \ll \hat{F}$, in the sense of~\eqref{major1} (with $E=\R^p$ and
$|\,.\,|_E$ the norm given by the sum of the absolute values of the
components).

Now, given $\alpha \geq 1$ and $s>0$, let us define the following formal power series in one variable
\begin{equation}\label{defM}
M_{\alpha,s}(X):=c^{-1}\sum_{l=0}^{+\infty}\frac{l!^{\alpha-1}}{(l+1)^2} \left(\frac{X}{s^\alpha}\right)^l=c^{-1}\sum_{l=0}^{+\infty}\frac{M_l}{l!}X^l, \quad M_l=\frac{l!^{\alpha}}{(l+1)^2 s^{\alpha l}}, \quad c=4\pi^2/3. 
\end{equation}
The following characterization of Gevrey functions is evident from the definitions~\eqref{defnorm} and~\eqref{major2}.

\begin{proposition}\label{propnorme}
  If $f : K \to \R^p$ is a smooth function,
  \[|f|_{\alpha,s} = \inf \left\{ C \in [0,+\infty]\, | \; f \ll_K C
  M_{\alpha,s}\right\}.\]
\end{proposition}

Henceforth $\alpha\geq 1$ will be fixed, so in the sequel we will
simply write $M_{\alpha,s}=M_s$.

\subsection{Properties of majorant series}

We collect here some properties of majorant series that will be used later on. It is clear how to define the derivatives of a formal power series in one variable, and also a linear combination and the product of two such formal power series. We then have the following lemma.

\begin{lemma}\label{properties} 
  Let $f,g : K \rightarrow \R^p$ be smooth functions,
  $F,G$ be two formal power series in one variable, and assume that
\[ f \ll_K F, \quad g \ll_K G.  \]
Then
\begin{equation}\label{deriv}
\partial^k f \ll_K \partial^{|k|} F, \quad k \in \N^m,
\end{equation}
\begin{equation}\label{somme}
\lambda f + \mu g \ll_K |\lambda|F+|\mu|G, \quad \lambda \in \R, \; \mu \in \R.
\end{equation}
Moreover, if we define $f\cdot g : K \rightarrow \R$ by 
\[ f\cdot g=\sum_{i=1}^pf_ig_i, \] 
then
\begin{equation}\label{prod}
f\cdot g \ll_K FG.
\end{equation}
\end{lemma}

For a proof, we refer to~\cite{SCK03}, Lemma $2.2$, in which the case $p=1$ is considered; but the general case $p \geq 1$ is entirely similar.

Given two formal power series in one variable $F$ and $G$, we define the composition $F \odot G$ of $F$ and $G$ by
\[ F \odot G(X):=\sum_{l=0}^{+\infty}\frac{F_l}{l!}\left(G(X)-G(0)\right)^l. \]

\begin{lemma}\label{comp}
Let $f : K \rightarrow \R^p$ be a smooth function, $g : L \rightarrow \T^{m_1} \times \R^{m_2}$ another smooth function such that $g(L)\subseteq K$, and assume that
\[ f \ll_K F, \quad g \ll_{L} G. \]
Then
\[ f \circ g \ll_{L} F \odot G.  \]
\end{lemma}

Once again, for a proof we refer to~\cite{SCK03}, Lemma $2.3$.

\subsection{Derivatives}

In this section, we will show that the derivatives of a Gevrey function are still Gevrey, at the expense of reducing the parameter $s>0$; these are analogues of Cauchy estimates for analytic functions.

\begin{proposition}\label{derivative}
Let $f \in G_{\alpha,s}(K,\R^p)$, and $0<\sigma<s$. Then for any $k \in \N^m$, $\partial^k f \in G_{\alpha,s-\sigma}(K,\R^p)$ 	and we have
\[ |\partial^k f|_{\alpha,s-\sigma} \leq \left(\frac{|k|^\alpha}{\sigma^\alpha}\right)^{|k|}|f|_{\alpha,s}. \] 
\end{proposition}

\begin{proof}
It is enough to prove the case $|k|=1$, as the general case follows by an easy induction. From Proposition~\ref{propnorme} and~\eqref{deriv} of Lemma~\ref{properties}, to prove the case $|k|=1$ it is sufficient to prove that
\begin{equation}\label{aprouver}
\partial^1 M_s \ll \sigma^{-\alpha}M_{s-\sigma}
\end{equation}
where $M_s$ is the formal power series defined in~\eqref{defM}. We have
\[ \partial^1
  M_{s}(X)=c^{-1}\sum_{l=1}^{+\infty}\frac{l!^{\alpha}}{(l+1)^2
    s^{l\alpha}(l-1)!}
  X^{l-1}=c^{-1}\sum_{l=0}^{+\infty}\frac{(l+1)!^{\alpha}}{(l+2)^2
    s^{(l+1)\alpha}l!} X^l \] and hence~\eqref{aprouver} is true if, for
all $l \in \N$,
\begin{equation}\label{aprouver2}
\left(\sigma s^{-1}\right)^{\alpha}\left(1-\sigma s^{-1}\right)^{l\alpha}(l+1)^\alpha \leq 1.
\end{equation}
Since $0 < \sigma s^{-1} <1$, we have $\ln(1-\sigma s^{-1}) \leq -\sigma s^{-1}$ and thus
\[ \left(1-\sigma s^{-1}\right)^{l\alpha}(l+1)^\alpha=e^{l\alpha\ln(1-\sigma s^{-1})}(l+1)^\alpha \leq e^{-l\alpha\sigma s^{-1}}(l+1)^\alpha.   \]
Let $\lambda=\sigma s^{-1}$, and consider the function $u(x)=(x+1)^\alpha e^{-\alpha \lambda x}$ for $x \geq 0$. This function reaches its maximum at $x=\lambda^{-1}(1-\lambda)$, the value of which is
\[ u\left(\lambda^{-1}(1-\lambda)\right)=\lambda^{-\alpha}e^{(\lambda-1)\alpha} \leq \lambda^{-\alpha}. \]
Therefore, for all $l \in \N$, we have 
\[ \left(1-\sigma s^{-1}\right)^{l\alpha}(l+1)^\alpha \leq \lambda^{-\alpha}=\left(\sigma s^{-1}\right)^{-\alpha} \]
which is exactly the inequality~\eqref{aprouver2} we wanted to prove.
\end{proof}

For $f : K \rightarrow \R$, let $\nabla f : K \rightarrow \R^m$ be the vector-valued function formed by the partial derivatives of $f$ of order one, and more generally, for $f : K \rightarrow \R^p$, we let $\nabla f : K \rightarrow M_{m,p}(R) \simeq \R^{mp}$ be the matrix-valued function whose columns are given by $\nabla f_i$ where $f=(f_i)_{1 \leq i \leq p}$. Then we have the following obvious corollary of Proposition~\ref{derivative}.

\begin{corollary}\label{corderivative}
Let $0<\sigma<s$. If $f \in G_{\alpha,s}(K,\R)$, then $\nabla f \in G_{\alpha,s-\sigma}(K,\R^m)$ and
\[ |\nabla f|_{\alpha,s-\sigma} \leq m\sigma^{-a}|f|_{\alpha,s}\]
and if $f \in G_{\alpha,s}(K,\R^p)$, then $\nabla f \in G_{\alpha,s-\sigma}(K,\R^{mp})$ and
\[ |\nabla f|_{\alpha,s-\sigma} \leq mp\sigma^{-a}|f|_{\alpha,s}.\] 
\end{corollary}

\subsection{Products}

In this section, we shall prove that the product of Gevrey functions is still a Gevrey function.

\begin{proposition}\label{produit}
Let $f,g \in G_{\alpha,s}(K,\R^p)$. Then $f\cdot g \in G_{\alpha,s}(K,\R)$ and we have
\[ |f\cdot g|_{\alpha,s} \leq |f|_{\alpha,s}|g|_{\alpha,s}. \]
\end{proposition}

Once again, in view of Proposition~\ref{propnorme} and~\eqref{prod} of Proposition~\ref{properties}, Proposition~\ref{produit} is a direct consequence of the following lemma.

\begin{lemma}\label{lemmeprod}
We have
\[ M_s^2 \ll M_s. \]
\end{lemma}

The proof given below follows~\cite{Lax53}. It is this lemma that motivates the introduction of the normalizing constant in $M_s$ (and thus in the Gevrey norm); without this constant one would have $M_s^2 \ll c^2 M_s$. Let us point out that the proof given below is elementary thanks to the factor $(|k|+1)^2$ in the definition of $M_s$; without this factor, the statement is true (with a different normalizing constant) but the proof is more involved (see Lemma $2.7$ of~\cite{SCK03}). 

\begin{proof}
Recall that
\[ M_{s}(X)=c^{-1}\sum_{l=0}^{+\infty}\frac{M_l}{l!}X^l, \quad M_l=\frac{l!^{\alpha}}{(l+1)^2 s^{\alpha l}} \]
and so the assertion of the lemma amounts to prove that for all $l \in \N$,
\begin{equation}\label{toprove}
\sum_{j=0}^l\frac{(j!)^{\alpha-1}((l-j)!)^{\alpha-1}}{(j+1)^2(l-j+1)^2} \leq c \frac{(l!)^{\alpha-1}}{(l+1)^2}, \quad c=\frac{4\pi^2}{3}. 
\end{equation}
Observe that the sum in the left-hand side of~\eqref{toprove} is symmetric with respect to $j \mapsto l-j$, and that since $\alpha-1\geq 0$, $(j!)^{\alpha-1}((l-j)!)^{\alpha-1} \leq (l!)^{\alpha-1}$ for all $l \in \N$. Hence,
\[  \sum_{j=0}^l\frac{(j!)^{\alpha-1}((l-j)!)^{\alpha-1}}{(j+1)^2(l-j+1)^2} \leq 2\sum_{j=0}^{l/2}\frac{(j!)^{\alpha-1}((l-j)!)^{\alpha-1}}{(j+1)^2(l-j+1)^2} \leq 2(l!)^{\alpha-1} \sum_{j=0}^{l/2}\frac{1}{(j+1)^2(l-j+1)^2}.  \]
Then for any $0 \leq j \leq l/2$, $(l-j+1)^2 \geq (l/2+1)^2 \geq (l+1)^2/4$, and therefore 
\[  \sum_{j=0}^l\frac{(j!)^{\alpha-1}((l-j)!)^{\alpha-1}}{(j+1)^2(l-j+1)^2} \leq \frac{8(l!)^{\alpha-1}}{(l+1)^2}\sum_{j=0}^{l/2}\frac{1}{(j+1)^2} \leq \frac{8(l!)^{\alpha-1}}{(l+1)^2}\sum_{j=0}^{+\infty}\frac{1}{(j+1)^2}=\frac{4\pi^2}{3}\frac{(l!)^{\alpha-1}}{(l+1)^2} \]
which is the inequality we wanted to prove.
\end{proof}

Proposition~\ref{produit} can be extended to matrix-valued functions. More precisely, given $f : K \rightarrow M_{m,p}(\R)$ and $g : K \rightarrow M_{p,q}(\R)$ where $f=(f_{i,j})_{1 \leq i \leq m, \; 1 \leq j \leq p}$ and $g=(g_{j,k})_{1 \leq j \leq m, \; 1 \leq k \leq p}$, we define $f \cdot g : K \rightarrow M_{m,q}(\R)$ by $f\cdot g:=((f\cdot g)_{i,k})_{1 \leq i \leq m, \; 1 \leq k \leq q}$ where
\[ (f\cdot g)_{i,k}:=\sum_{j=1}^p f_{i,j}g_{j,k}. \]
Then the following statement is an obvious corollary of Proposition~\ref{produit}.

\begin{corollary}\label{corproduit}
Let $f \in G_{\alpha,s}(K,M_{m,p}(\R))$, $g \in G_{\alpha,s}(K,M_{p,q}(\R))$. Then $f\cdot g \in G_{\alpha,s}(K,M_{m,q}(\R))$ and we have
\[ |f\cdot g|_{\alpha,s} \leq |f|_{\alpha,s}|g|_{\alpha,s}. \]
\end{corollary}

\subsection{Compositions}

Our goal here is to prove that the composition of Gevrey functions are still Gevrey. But first we need to define two additional formal power series associated to $M_s$, the latter being defined in~\eqref{defM}. So we define $\bar{M}$ by
\begin{equation}\label{defbarM}
\bar{M}_s(X):=M_s(X)-M_s(0)=c^{-1}\sum_{l=1}^{+\infty}\frac{l!^{\alpha-1}}{(l+1)^2} \left(\frac{X}{s^\alpha}\right)^l
\end{equation}
and $\tilde{M}$ by
\begin{equation}\label{deftildeM}
\tilde{M}_s(X):=c^{-1}\sum_{l=0}^{+\infty}\frac{(l+1)!^{\alpha-1}}{(l+2)^2} \left(\frac{X}{s^\alpha}\right)^l.
\end{equation}
It is clear that
\begin{equation}\label{barMtildeM}
s^{-\alpha}X\tilde{M}_s(X)=\bar{M}_s(X).
\end{equation}

\begin{lemma}\label{lemmecomp}
We have
\[ \tilde{M}_s^2 \ll \tilde{M}_s, \quad \tilde{M}_s\bar{M}_s \ll  \bar{M}_s. \]
\end{lemma}

As for Lemma~\ref{lemmeprod}, the factor $(|k|+1)^2$ in the definition
of $M_s$ makes the proof simple, but the statement is still true with
this factor (see Lemma $2.4$ of~\cite{SCK03}).

\begin{proof}
It is enough to prove the first part of the statement, as the second part of the statement follows from it; indeed, if $\tilde{M}_s^2 \ll \tilde{M}_s$, then using~\eqref{barMtildeM} we have
\[ \tilde{M}_s(X)\bar{M}_s(X)=s^{-\alpha}X\tilde{M}_s(X)\tilde{M}_s(X) \ll s^{-\alpha}X\tilde{M}_s(X)=\bar{M}_s(X). \]
As in Lemma~\ref{lemmeprod}, to prove that $\tilde{M}_s^2 \ll \tilde{M}_s$ one needs to show
\begin{equation}\label{toprove2}
\sum_{j=0}^l\frac{((j+1)!)^{\alpha-1}((l-j+1)!)^{\alpha-1}}{(j+2)^2(l-j+2)^2} \leq c \frac{((l+1)!)^{\alpha-1}}{(l+2)^2}, \quad c=\frac{4\pi^2}{3}. 
\end{equation}
The sum in the left-hand side of~\eqref{toprove2} is still symmetric with respect to $j \mapsto l-j$, and therefore
\[
\sum_{j=0}^l\frac{((j+1)!)^{\alpha-1}((l-j+1)!)^{\alpha-1}}{(j+2)^2(l-j+2)^2}
\leq
2\sum_{j=0}^{l/2}\frac{((j+1)!)^{\alpha-1}((l-j+1)!)^{\alpha-1}}{(j+2)^2(l-j+2)^2}.  \]
Then, for any $l \in \N$ and any $0 \leq j \leq l/2$, since
$\alpha-1 \geq 0$ we have
\[ ((j+1)!)^{\alpha-1}((l-j+1)!)^{\alpha-1} \leq ((l+1)!)^{\alpha-1} \]
as one may easily check. Then, as in Lemma~\ref{lemmeprod}, for any $0 \leq j \leq l/2$, $(l-j+2)^2 \geq (l/2+2)^2 \geq (l+2)^2/4$, and therefore 
\[  \sum_{j=0}^l\frac{((j+1)!)^{\alpha-1}((l-j+1)!)^{\alpha-1}}{(j+2)^2(l-j+2)^2} \leq \frac{8((l+1)!)^{\alpha-1}}{(l+1)^2}\sum_{j=0}^{l/2}\frac{1}{(j+2)^2} \leq \frac{4\pi^2}{3} \frac{((l+1)!)^{\alpha-1}}{(l+2)^2}\]
and this concludes the proof.
\end{proof}

\begin{proposition}\label{composition}
  Let $f \in G_{\alpha,s}(K,\R^p)$, $0 < \sigma < s$,
 and $g \in G_{\alpha,s-\sigma}(L,\T^{m_1} \times \R^{m_2})$ such that $g(L)\subseteq K$. Assume that
  $g=\mathrm{Id}+u$ with
  \begin{equation}\label{small}
    |u|_{\alpha,s-\sigma} \leq \sigma^\alpha.
  \end{equation}  
  Then $f \circ g \in G_{\alpha,s-\sigma}(L,\R^p)$ and
  \begin{equation}
    \label{eq:composition}%
    |f \circ g|_{\alpha,s-\sigma} \leq |f|_{\alpha,s}.
  \end{equation}
\end{proposition}

When $f$ et $g$ are analytic, the analogue of
estimate~\eqref{eq:composition} with the supremum norm is obvious, for
the obvious reason that the supremum of a function is a non-increasing
function of the domain.

As it will be clear in the proof, the conclusions of
Proposition~\ref{composition} holds true under the slightly weaker
assumption that
\[ |u|_{\alpha,s-\sigma}-\sup_{a \in K}|u(a)| \leq s^\alpha-(s-\sigma)^\alpha \]
but this will not be needed.

\begin{proof}
Let 
\[ a:=|f|_{\alpha,s}, \quad b:=|u|_{\alpha,s-\sigma} \]
so that, from Proposition~\ref{propnorme}, 
\[ f(x) \ll_K aM_s(X), \quad u(x) \ll_L bM_{s-\sigma}(X) \]
and consequently
\[ f(x) \ll_K aM_s(X), \quad g(x) \ll_L X+bM_{s-\sigma}(X). \]
We now apply Lemma~\ref{comp} and, recalling the definition of $\bar{M}_s$ and $\tilde{M}_s$ given respectively in~\eqref{defbarM} and~\eqref{deftildeM}, we obtain
\begin{eqnarray}
f(g(x)) & \ll_L & aM_s\left(X+b\bar{M}_{s-\sigma}(X)\right) \nonumber \\ \nonumber
& = & aM_s(0)+a\bar{M}_s\left(X+b\bar{M}_{s-\sigma}(X)\right) \\ \nonumber
& = & aM_s(0)+a\bar{M}_s\left(X+b(s-\sigma)^{-\alpha}X\tilde{M}_{s-\sigma}(X)\right) \\ \nonumber
& = & aM_s(0)+a\sum_{l=1}^{+\infty}\frac{l!^{\alpha-1}}{(l+1)^2} \left(\frac{X+b(s-\sigma)^{-\alpha}X\tilde{M}_{s-\sigma}(X)}{s^\alpha}\right)^l \\ \label{estder}
& = & aM_s(0)+a\sum_{l=1}^{+\infty}\frac{l!^{\alpha-1}}{(l+1)^2} \left(\frac{X}{s^\alpha}\right)^l\left(1+b(s-\sigma)^{-\alpha}\tilde{M}_{s-\sigma}(X)\right)^l.
\end{eqnarray}
From the first part of Lemma~\ref{lemmecomp}, for any $j \in \N$, we have
\[ \tilde{M}_{s-\sigma}^j \ll \tilde{M}_{s-\sigma} \]
and therefore
\begin{eqnarray*}
\left(1+b(s-\sigma)^{-\alpha}\tilde{M}_{s-\sigma}(X)\right)^l
& = & \sum_{j=0}^l \binom{l}{j}b^j(s-\sigma)^{-j\alpha}\tilde{M}_{s-\sigma}(X)^j \\
& \ll & \tilde{M}_{s-\sigma}(X)\sum_{j=0}^l \binom{l}{j}b^j(s-\sigma)^{-j\alpha} \\
& = & \tilde{M}_{s-\sigma}(X)\left(1+b(s-\sigma)^{-\alpha}\right)^l.
\end{eqnarray*}
Now, from~\eqref{small} we get
\[ b=|u|_{\alpha,s-\sigma} \leq \sigma^\alpha \leq s^\alpha-(s-\sigma)^\alpha\]
and thus
\[ 1+ b(s-\sigma)^{-\alpha} \leq s^\alpha(s-\sigma)^{-\alpha}.  \]
This gives
\[ \left(1+b(s-\sigma)^{-\alpha}\tilde{M}_{s-\sigma}(X)\right)^l \ll \tilde{M}_{s-\sigma}(X)(s^\alpha(s-\sigma)^{-\alpha})^l  \]
which, together with~\eqref{estder}, yields
\[ f(g(x))  \ll_L aM_s(0)+a\tilde{M}_{s-\sigma}(X)\sum_{l=1}^{+\infty}\frac{l!^{\alpha-1}}{(l+1)^2} \left(\frac{X}{(s-\sigma)^\alpha}\right)^l=aM_s(0)+a\tilde{M}_{s-\sigma}(X)\bar{M}_{s-\sigma}(X).\]
Using the second part of Lemma~\ref{lemmecomp}, this gives
\[ f(g(x))  \ll_L aM_s(0)+a\bar{M}_{s-\sigma}(X)\]
and since $M_s(0)=M_{s-\sigma}(0)$, we arrive at
\[ f(g(x))  \ll_L a(M_{s-\sigma}(0)+\bar{M}_{s-\sigma}(X))=aM_{s-\sigma}(X).\]
Using Proposition~\ref{propnorme}, we eventually obtain
\[ |f \circ g|_{\alpha,s-\sigma} \leq a= |f|_{\alpha,s} \]
and this concludes the proof.
\end{proof}

\subsection{Flows}

In this section and the next one, we shall state and prove some
estimates adapted to the situation considered in~\S\ref{sec:proof}:
that is we consider functions $H=H(\theta,I,\omega)$ which are defined
and Gevrey smooth on a domain of the form
\[ \T^n \times D_{r,h}=\T^n \times D_r \times D_h \subseteq \T^n \times \R^n \times \R^n \]
where $D_r$ is the ball of radius $r>0$ centered at the origin and $D_h$ is an arbitrary ball of radius $h>0$. In the lemma and proposition below, the variables $\omega \in D_h$ play the role of a fixed parameter, hence to simplify the notations we will explicitly suppress the dependence on $\omega \in D_h$. 

Moreover, throughout this section and the next one, for simplicity we shall write $u \MP v$ (respectively $u \PM v$), if, for some constant $C\geq 1$ which depends only on $n$ and $\alpha$ and could be made explicit, we have $u\leq Cv$ (respectively $Cu \leq v$).

Let us first start with a vector-valued function $D: \T^n \rightarrow \R^n$ which depends only on $\theta \in \T^n$, and that we shall considered as a vector field on $\T^n$.

\begin{lemma}\label{flot}
Given $D \in G_{\alpha,s}(\T^n,\R^n)$, let $0 < \sigma < s$ and assume that
\begin{equation}\label{small2}
|D|_{\alpha,s} \PM \sigma^\alpha.
\end{equation}
Then for any $t \in [0,1]$, the time-$t$ map $D^t$ of the flow of $D$ belongs to $G_{\alpha,s-\sigma}(\T^n,\T^n)$ and we have the estimate
\begin{equation}\label{estflot}
|D^t-\mathrm{Id}|_{\alpha,s-\sigma} \leq |D|_{\alpha,s}.
\end{equation} 
\end{lemma} 

The proof of the above lemma is a variant of the proof of Lemma B.3 in~\cite{LMS16}.

\begin{proof}
The fact that $D^t$ is smooth and defined for all $t \in [0,1]$ (in fact, for all $t \in \R$) follows from the compactness of $\T^n$ and the classical result on the existence and uniqueness of solutions of differential equations (even though this will essentially be re-proved below); the only thing we need to prove is the estimate~\eqref{estflot}. So let us consider the space $V:=C([0,1],G_{\alpha,s-\sigma}(\T^n,\T^n))$ of continuous map from $[0,1]$ to $G_{\alpha,s-\sigma}(\T^n,\T^n)$: given an element $\Phi \in V$ and $t \in [0,1]$, we shall write $\Phi^t:=\Phi(t)$ and consequently $\Phi=(\Phi^t)_{t \in [0,1]}$. We equip $V$ with the following norm:
\[ ||\Phi||:=\max_{t \in [0,1]}|\Phi^t|_{\alpha,s-\sigma} \]
which makes it a Banach space, and if we set $\rho:=|D|_{\alpha,s}$, we define
\[ B_\rho V:=\{ \Phi \in V \; | \; ||\Phi-\mathrm{Id}|| \leq \rho \}.  \]
We can eventually define a Picard operator $P$ associated to $D$ by
\[ P : B_\rho V \rightarrow B_\rho V, \quad \Phi \mapsto P(\Phi) \]
where $P(\Phi)=(P(\Phi)^t)_{t \in [0,1]}$ is defined by
\[ P(\Phi)^t:=\mathrm{Id}+\int_0^t D \circ \Phi^\tau d\tau. \]
To prove the lemma, it is sufficient to prove that $P$ has a unique fixed point $\Phi_* \in B_\rho V$, as necessarily $(\Phi_*^t)_{t \in [0,1]}=(D^t)_{t \in [0,1]}$. Therefore it is sufficient to prove that $P$ induces a well-defined contraction on $B_\rho V$, as the latter is a complete subset of the Banach space $V$.

First we need to show that $P$ maps $B_\rho V$ into itself. So assume $\Phi \in B_\rho V$, using~\eqref{small2} this implies that for all $t \in [0,1]$, 
\[ |\Phi^t-\mathrm{Id}|_{\alpha,s-\sigma} \leq \rho \leq \sigma^\alpha \]
so that~\eqref{small} of Proposition~\ref{composition} is satisfied (with $f=D$ and $g=\Phi^t$ for any $t \in [0,1]$) and the latter proposition applies: this gives
\[ |D \circ \Phi^t|_{\alpha,s-\sigma} \leq |D|_{\alpha,s}=\rho, \quad t \in [0,1] \]
hence
\[ \left|P(\Phi)^t-\mathrm{Id}\right|_{\alpha,s-\sigma}=\left|\int_0^t D \circ \Phi^\tau d\tau\right|_{\alpha,s-\sigma} \leq t\rho \leq \rho, \quad t \in [0,1] \]
and therefore
\[ ||P(\Phi)-\mathrm{Id}|| \leq \rho. \]
This proves that $P$ maps $B_\rho V$ into itself.

It remains to show that $P$ is a contraction. So let $\Phi_1,\Phi_2 \in B_\rho V$, then for any $t \in [0,1]$,
\begin{eqnarray*}
 P(\Phi_1)^t-P(\Phi_2)^t & = & \int_0^t \left(D \circ \Phi_1^\tau-D \circ \Phi_2^\tau \right)d\tau \\
& = & \int_0^t \left(\int_0^1 \nabla D \circ (s\Phi_1^\tau+(1-s)\Phi_2^\tau)ds \right) \cdot (\Phi_1^\tau-\Phi_2^\tau) d\tau. 
 \end{eqnarray*}
Using Proposition~\ref{composition}, Corollary~\ref{corderivative} and Corollary~\ref{corproduit}, we obtain, for any $t \in [0,1]$,
\[ |P(\Phi_1)^t-P(\Phi_2)^t|_{\alpha,s-\sigma} \MP \sigma^{-\alpha}|D|_{\alpha,s}\max_{0 \leq \tau \leq t}|\Phi_1^\tau-\Phi_2^\tau|_{\alpha,s-\sigma} \MP \sigma^{-\alpha}|D|_{\alpha,s}||\Phi_1-\Phi_2||   \]
and hence
\[ ||P(\Phi_1)-P(\Phi_2)|| \MP \sigma^{-\alpha}|D|_{\alpha,s}||\Phi_1-\Phi_2||.  \]
Using~\eqref{small2}, we can then insure that $P$ is a contradiction, which concludes the proof.  
\end{proof}

Now let us consider a Hamiltonian function $X$ on $\T^n \times D_r$, of the form
\begin{equation}\label{HamX}
X(\theta,I):=C(\theta)+D(\theta) \cdot I, \quad  C:\T^n \rightarrow \R, \quad D : \T^n \rightarrow \R^n.
\end{equation}
The Hamiltonian equations associated to $X$ are given by:
\[ 
\begin{cases}
\dot{\theta}(t)=\nabla_I X(\theta(t),I(t))=D(\theta(t)), \\
\dot{I}(t)=-\nabla_\theta X(\theta(t),I(t))=-\nabla C(\theta(t))-\nabla D(\theta(t))\cdot I.
\end{cases}
\]
The equations for $\theta$ are uncoupled from the equations of $I$ (and hence can be integrated independently), while the equations for $I$ are affine in $I$; it is well-known that these facts lead to a simple form of the Hamiltonian flow associated to $X$ (see, for instance,~\cite{Vil08}).

\begin{proposition}\label{flotX}
Let $X$ be as in~\eqref{HamX} with $C \in G_{\alpha,s}(\T^n,\R)$ and $D \in G_{\alpha,s}(\T^n,\R^n)$. Let $0 < \sigma < s$ and assume that
\begin{equation}\label{smallX}
|D|_{\alpha,s} \PM \sigma^\alpha.
\end{equation}
Then for any $t \in [0,1]$, the time-$t$ map $X^t$ of the Hamiltonian flow of $X$ is of the form
\[ X^t(\theta,I)=(\theta+E^t(\theta),I+F^t(\theta)\cdot I+G^t(\theta)) \]
where $E^t \in G_{\alpha,s-\sigma}(\T^n,\R^n)$, $F^t \in G_{\alpha,s-\sigma}(\T^n,\R^{n^2})$ and $G^t \in G_{\alpha,s-\sigma}(\T^n,\R^n)$ with the estimates
\begin{equation}\label{estflotX}
|E^t|_{\alpha,s-\sigma} \leq |D|_{\alpha,s}, \quad |F^t|_{\alpha,s-\sigma} \MP \sigma^{-\alpha}|D|_{\alpha,s}, \quad |G^t|_{\alpha,s-\sigma} \MP \sigma^{-\alpha}|C|_{\alpha,s}. 
\end{equation} 
As a consequence, given $0<\delta<r$, if we further assume that 
\begin{equation}\label{smallXX}
r\sigma^{-\alpha}|D|_{\alpha,s}+\sigma^{-\alpha}|C|_{\alpha,s} \PM \delta 
\end{equation}
then $X^t$ maps $\T^n \times D_{r-\delta}$ into $\T^n \times D_{r}$. 
\end{proposition} 

\begin{proof}
The second part of the statement clearly follows from the first part, so let us prove the latter. From the specific form of the Hamiltonian equations associated to $X$, one has, for any $t \in [0,1]$,
\[ X^t(\theta,I)=(\theta+E^t(\theta),I+F^t(\theta)\cdot I+G^t(\theta)) \]
with
\[  
\begin{cases}
E^t(\theta)=\int_{0}^t D(\theta+E^\tau(\theta))d\tau, \\
F^t(\theta)=-\int_{0}^t \nabla D(\theta+E^\tau(\theta))d\tau-\int_{0}^t \nabla D(\theta+E^\tau(\theta))\cdot F^\tau(\theta)d\tau \\
G^t(\theta)=-\int_{0}^t \nabla C(\theta+E^\tau(\theta))d\tau-\int_{0}^t \nabla D(\theta+E^\tau(\theta))\cdot G^\tau(\theta)d\tau. 
\end{cases}
\]
Because of~\eqref{smallX}, Lemma~\ref{flot} applies and the flow $D^t(\theta)=\theta+E^t(\theta)$ satisfies~\eqref{estflot}, and therefore 
\[|E^t|_{\alpha,s-\sigma}=|D^t-\mathrm{Id}|_{\alpha,s-\sigma} \leq |D|_{\alpha,s}\]
which gives the first estimate of~\eqref{flotX}. Using this estimate and~\eqref{smallX}, we can apply Proposition~\ref{composition} and Corollary~\ref{corderivative} (both with $\sigma/2$ instead of $\sigma$) to obtain, for any $0 \leq \tau \leq  t \leq 1$,
\[ |\nabla D \circ D^\tau|_{\alpha,s-\sigma} \leq |\nabla D|_{\alpha,s-\sigma/2} \MP \sigma^{-\alpha}|D|_{\alpha,s}.   \]
Looking at the expression of $F^t$, this gives
\[ |F^t|_{\alpha,s-\sigma} \MP \sigma^{-\alpha}|D|_{\alpha,s}\left(1+\int_0^t |F^\tau|_{\alpha,s-\sigma} d\tau \right) \]
which, by Gronwall's inequality and~\eqref{smallX}, implies that for all $t \in [0,1]$,
\[ |F^t|_{\alpha,s-\sigma} \MP \sigma^{-\alpha}|D|_{\alpha,s} \]
which is the second estimate of~\eqref{flotX}. For the third estimate of~\eqref{flotX}, observe that the same argument yields
\[ |G^t|_{\alpha,s-\sigma} \MP \sigma^{-\alpha}|C|_{\alpha,s}+\sigma^{-\alpha}|D|_{\alpha,s}\int_0^t |G^\tau|_{\alpha,s-\sigma} d\tau  \]
and again, by Gronwall's inequality and~\eqref{smallX}, for all $t \in [0,1]$ we have
\[ |G^t|_{\alpha,s-\sigma} \MP \sigma^{-\alpha}|C|_{\alpha,s}. \]
This concludes the proof.
\end{proof}

\subsection{Inverse functions}

In this last section, we shall prove that if a Gevrey map is
sufficiently close to the identity, then its local inverse is still
Gevrey. To prove this in a setting adapted to~\S\ref{sec:proof}, let
us consider a map $\phi$ which depends only on $\omega \in D_h$, that
is $\phi : D_h \rightarrow \R^n$.

\begin{proposition}\label{propomega}
  Given $\phi \in G_{\alpha,s}(D_h,\R^n)$, let $0 < \sigma < s$ and
  assume that
  \begin{equation}\label{smallomega}
    |\phi-\mathrm{Id}|_{\alpha,s} \PM \sigma^{\alpha}, \quad
    |\phi-\mathrm{Id}|_{\alpha,s} \leq h/2  
  \end{equation}
  Then there exists a unique
  $\varphi \in G_{\alpha,s-\sigma}(D_{h/2},D_h)$ such that
  $\phi \circ \varphi=\mathrm{Id}$ and
  \begin{equation}\label{estomega}
    |\varphi-\mathrm{Id}|_{\alpha,s-\sigma} \leq |\phi-\mathrm{Id}|_{\alpha,s}. 
  \end{equation}
\end{proposition}

\begin{proof}
  Let us define $V:=G_{\alpha,s-\sigma}(D_{h/2},\R^n)$, which is a
  Banach space with the norm $||\,.\,||=|\,.\,|_{\alpha,s-\sigma}$,
  and for $\rho:=|\phi-\mathrm{Id}|_{\alpha,s}$, we set
\[ B_\rho V:=\{\psi \in V \; | \; ||\psi-\mathrm{Id}|| \leq \rho\}. \]
Let us define the following Picard operator $P$ associated to $\phi$:
\[ P : B_\rho V \rightarrow B_\rho V, \quad \psi \mapsto P(\psi)=\mathrm{Id}-(\phi-\mathrm{Id})\circ\psi. \]
It is clear that $\phi \circ \varphi=\mathrm{Id}$ if and only if $\varphi$ is a fixed point of $P$, and therefore the proposition will be proved once we have shown that $P$ has a unique fixed point in $B_\rho V$, and to do this it is enough to prove that $P$ is a well-defined contraction of $B_\rho V$.

First let us prove that $P$ maps $B_\rho V$ into itself. So let $\psi \in B_\rho V$, and using the second part of~\eqref{smallomega}, observe that since
\[ \sup_{\omega \in D_{h/2}}|\psi(\omega)-\omega|\leq ||\psi-\mathrm{Id}||\leq \rho \leq h/2  \]
then $\psi$ maps $D_{h/2}$ into $D_h$. This, together with the first part of~\eqref{smallomega} allows us to apply Proposition~\ref{composition} to get
\[ ||(\phi-\mathrm{Id})\circ\psi||=|(\phi-\mathrm{Id})\circ\psi|_{\alpha,s-\sigma} \leq |\phi-\mathrm{Id}|_{\alpha,s}=\rho \]
and thus 
\[ ||P(\psi)-\mathrm{Id}||=||(\phi-\mathrm{Id})\circ\psi|| \leq \rho, \]
that is, $P$ maps $B_\rho V$ into itself. To show that $P$ is a contraction, using Corollary~\ref{corderivative}, Corollary~\ref{corproduit} and Proposition~\ref{composition} one gets, for any $\psi_1,\psi_2 \in B_\rho V$:
\[ ||(\phi-\mathrm{Id})\circ\psi_1-(\phi-\mathrm{Id})\circ\psi_2|| \MP \sigma^{-\alpha} |\phi-\mathrm{Id}|_{\alpha,s} ||\psi_1-\psi_2|| \] 
and from the first part of~\eqref{smallomega}, one can make sure that $P$ is a contraction. This ends the proof.  
\end{proof}

\textit{Comment.} After this work was made public on Arxiv, an independent and interesting proof of a special case of Theorem~\ref{KAMvector} appeared in the preprint~\cite{LDG17}.

\bigskip

\textit{Acknowledgements.} Part of this work was done when the
authors were hosted by ETH Z{\"u}rich during a visit to V. Kaloshin. The authors have also benefited
from partial funding from the ANR project Beyond KAM. 

\addcontentsline{toc}{section}{References}
\bibliographystyle{amsalpha}
\bibliography{KGBR5}

\end{document}